\documentclass{article}
\usepackage[a4paper,top=3.cm,bottom=3.cm,left=2.5cm,right=2.5cm]{geometry}

\usepackage{amsthm}
\usepackage{graphicx}%
\usepackage{multirow}%
\usepackage{amssymb}
\usepackage{array}
\newcolumntype{P}[1]{>{\centering\arraybackslash}p{#1}}
\usepackage{thmtools, thm-restate}
\usepackage{mathrsfs}%
\usepackage[toc,page]{appendix}%
\usepackage{xcolor}%
\usepackage{textcomp}%
\usepackage{manyfoot}%
\usepackage{booktabs}%
\usepackage{listings}%
\usepackage{bm}%
\usepackage{listings}
\usepackage{yhmath}
\usepackage{mathtools}
\usepackage{mathrsfs}
\usepackage[utf8]{inputenc}
\usepackage{tikz}
\usepackage{setspace}
\usepackage{authblk}
\usepackage{enumitem}

\DeclareMathOperator*{\spn}{span}
\newcommand{\muv}{\boldsymbol{\mu}}
\newcommand{\xiv}{\boldsymbol{\xi}}
\newcommand{\bx}{\mathbf{x}}
\newcommand{\Omegaref}{\tilde{\Omega}}

\usepackage{algorithm}
\usepackage{algorithmic}

\newtheorem{lemma}{Lemma}[section]
\newtheorem{proposition}{Proposition}[section]

\newtheorem{assumption}{Assumption}[section]
\theoremstyle{remark}

\usepackage[colorlinks=true,linkcolor=black,anchorcolor=black,citecolor=black,filecolor=black,menucolor=black,runcolor=black,urlcolor=black]{hyperref} 
\usepackage{cleveref}
\usepackage[backend=biber]{biblatex}
\appto{\bibsetup}{\sloppy}
\usepackage{verbatim}

\usepackage{graphicx}
\usepackage[labelsep=period, labelfont=bf]{caption}
\usepackage{subcaption}
\usepackage{colortbl}
\usepackage{multirow}
\usepackage{tabularray}

\title{Handling geometrical variability in nonlinear reduced order modeling through Continuous Geometry-Aware DL-ROMs}

\begin{document}

\author[1]{Simone Brivio}
\author[1]{Stefania Fresca}
\author[1]{Andrea Manzoni}
\affil[1]{\normalsize MOX, Department of Mathematics, Politecnico di Milano, Milan, Italy}

\date{}
\maketitle
\vspace{-1.2cm}
\begin{center}
{\small \{\texttt{simone.brivio}, \texttt{stefania.fresca}, \texttt{andrea1.manzoni}\} \texttt{@polimi.it}}
\end{center}

\begin{abstract}
\noindent Deep Learning-based Reduced Order Models (DL-ROMs) provide nowadays a well-established class of accurate surrogate models for complex physical systems described by parametrized PDEs, by nonlinearly compressing the solution manifold into a handful of latent coordinates. Until now, design and application of DL-ROMs mainly focused on physically parameterized problems. Within this work, we provide a novel extension of these architectures to problems featuring geometrical variability and parametrized domains, namely, we propose Continuous Geometry-Aware DL-ROMs (CGA-DL-ROMs). In particular, the space-continuous nature of the proposed architecture matches the need to deal with \textit{multi-resolution} datasets, which are quite common in the case of geometrically parametrized problems. Moreover, CGA-DL-ROMs are endowed with a strong inductive bias that makes them aware of geometrical parametrizations, thus enhancing both the compression capability and the overall performance of the architecture. Within this work, we justify our findings through a thorough theoretical analysis, and we practically validate our claims by means of a series of numerical tests encompassing physically-and-geometrically parametrized PDEs, ranging from the unsteady Navier-Stokes equations for fluid dynamics to advection-diffusion-reaction equations for mathematical biology.
\end{abstract}
\begin{flushleft}
\end{flushleft}
  
\section{Introduction}
Representing one of the most widely used tools to describe complex physical systems from a mathematical viewpoint, Partial Differential Equations (PDEs) are nowadays of utmost importance in applied sciences and engineering. 
However, since their analytical solution is almost never available, several efforts were directed in the last decades towards the development of accurate numerical methods that strive to approximate the solution of PDEs by means of Full Order Models (FOMs). These latter might involve the discretization of the underlying governing equations on a mesh of step size $h >0$, thus yielding $N_h$ degrees of freedom (dofs). We emphasize that FOMs feature an outstanding accuracy, however requiring substantial computational resources, for instance when dealing with time-dependent problems, nonlinear terms, and in the case where the dimension $N_h$ of the corresponding algebraic problem is too large. This latter represents a major drawback, especially when dealing with \textit{many query} applications involving parametric PDEs -- like, e.g., uncertainty quantification and optimal control -- which entail the repeated solution of the numerical problem for different parameters' values.

\paragraph{ROMs: an historical perspective.} In the last decades, the need of \textit{real-time} simulations in \textit{many-query} frameworks paved the way to the development of Reduced Order Models (ROMs), which provide fast and efficient alternatives to FOMs while retaining the essential features of the underlying physical problem. Specifically, ROMs were originally proposed for problems depending on a set of physical parameters $\muv \in \mathcal{P}$ -- and possibly the time variable $t \in \mathcal{T}$ -- where $\mathcal{P}$ and $\mathcal{T}$ are compact finite-dimensional sets. Among the most adopted ROM techniques, we mention the reduced basis (RB) method \cite{quarteroni2016redb,benner2017model,benner2020model_1,benner2020model_2}, which seeks a low-dimensional representation of the solution manifold through linear subspaces. However, we mention that the trial manifold employed in a RB method is global and linear: for this reason, it might reproduce inefficiently the local features and the nonlinearities in the solution manifold. As a matter of fact, the computational speedup of RB methods (with respect to the corresponding FOM) is negligible in the case of real-world applications involving nonlinear and non-affine terms. More recently, the limitations of classical RB methods paved the way to the development of deep learning-based alternatives that include, but are not limited to, Neural Operators \cite{kovachki2023neural, lu2021learning, li2020fourier, wen2022ufno, lanthaler2023nonlocal, li2020multipole, li2020neural, cai2021deep, jin2022mionet}, Deep Learning-based ROMs (DL-ROMs) \cite{fresca2021comprehensive, fresca2022poddlrom, brivio2024ptpidlroms, pant2021deep, mucke2021reduced} and transformer architectures \cite{hao2023gnot, li2023transformer}. Both Neural Operators and DL-ROMs are equipped with a strong theoretical background that backs up their outstanding approximation capabilities \cite{lu2022comprehensive, lanthaler2022error, lanthaler2023nonlocal, kovachki2023neural, brivio2024error, franco2023deep}. On the other hand, even though transformers and attention-based mechanisms \cite{kissas2022learning,hemmasian2024multiscale,hao2023gnot} provide state-of-the-art prediction accuracy and grant us the possibility to tackle large datasets, they are less interpretable and not fully understood from a theoretical viewpoint as only some architectures are supplied with a comprehensive theoretical study \cite{cao2021choose,prasthofer2022variableinputdeepoperatornetworks}. Nonetheless, we stress that Neural Operators often consist of complex models featuring a high-dimensional latent space; conversely, DL-ROMs both accurately reproduce the variability of the solution manifold and entail a profound nonlinear dimensionality reduction, ultimately resolving the underlying dynamics into a handful of latent coordinates.

\paragraph{Surrogate models for geometrically parameterized problems.} Recently, many authors' focus shifted towards the development of ROMs for challenging differential problems characterized by geometrical parameters and parametric domain shapes. For instance, RB methods were extended to cope with geometrically parametrized problems \cite{rozza2008reduced,manzoni2012reduced, manzoni2012shape, lassila2010parametric, karatzas2022advanced, shah2020discontinuous, ballarin2016fast}. Alternatively, more efficient solutions were proposed in the context of deep learning-based ROMs. 
A first class of strategies involves the coupling of proper orthogonal decomposition (POD) and neural networks \cite{siena2023data}. However, the discrete nature of the approach limits the range of possible applications to diffeomorphic meshes showing the same resolution.
Other possible approaches make use of operator learning paradigms in infinite-dimensional function spaces, which are more suitable for parametric domains since they are mesh-agnostic.
For example, one may consider GNN approaches in an operator learning framework \cite{lotzsch2023learning,franco2023geom,pfaff2021learning}, which are very expressive and capable to reproduce the local features of the solution field with high precision. However, GNNs are data eager, that is, they need a large amount of training data to achieve a suitable accuracy on the test set. Moreover, their graph-based architecture entails a large memory footprint in the case of complex problems comprising high-dimensional data. 
Among other techniques, domain practitioners can take advantage of Implicit Neural Representations (INR), a class of approaches which is still at its infancy for applications involving PDEs and therefore is not equipped with theoretical results. We also remark that INR-based strategies such as \cite{serrano2023operator} may require inference-time optimization and are therefore expensive to evaluate in a \textit{many-query} context. For the sake of completeness, we note that some authors crafted suitable transformer-based models to deal with geometrically parametrized problems \cite{li2023transformer, hao2023gnot}. However, these strategies often involve complex and heavy architectures, and are also data eager, so that their accuracy is particularly sensitive to the number of available input-output pairs.
More remarkably, several approaches successfully adapted Neural Operators to handle geometrical parametrizations \cite{li2022fourier, li2023geometryinformed, tran2023factorized, liu2023domain}: the resulting frameworks although extending existing architectures, lack an inductive bias tailored for geometrically parametrized problems. For this reason, even though they provide strikingly accurate approximations, they often feature large neural network architectures, which are prone to a slow training and evaluation. On the other hand, DL-ROMs (and their extensions) usually entail smaller architectures, however they were originally conceived in the context of physically parametrized problems and they were never adapted to problems featuring geometrical parameters.

\paragraph{Main contributions and work outline.} The main contribution of our work consists in the extension of the classical DL-ROM architecture to parametric PDEs featuring geometrical variability. Specifically:
\begin{itemize}
    \item we justify the need of devising a space-continuous, infinite-dimensional paradigm to deal with \textit{multi-resolution} datasets encountered in problems featuring geometrical variability;
    \item we propose Continuous Geometry-Aware DL-ROMs (CGA-DL-ROMs), a comprehensive learning framework to construct DL-ROMs for geometrically parameterized problems, which {\em (i)} is based on a space-continuous formulation to account for \textit{multi-resolution} datasets, and {\em (ii)} leverages on geometry-aware basis functions that represent a major component yielding the notable approximation capabilities of the proposed architecture;
    \item we analyze the proposed framework, assessing its performance against suitable baselines both in light of the theory and by means of involved numerical experiments. 
\end{itemize}
The present work is organized as follows. In \S \ref{sec:formulation} we describe the general formulation of problems featuring geometrical variability; on top of that, in \S \ref{sec:infinite-dimensional-setting} we characterize CGA-DL-ROMs for geometrically parameterized problems. Moreover, in \S \ref{sec:modeling-basis-functions}  we describe the main features of CGA-DL-ROMs, focusing on the analysis of their peculiar geometry-aware basis functions, showing how to parametrize such basis functions through neural networks. We report in \S \ref{sec:numerical-experiments} a series of numerical experiments that support our theoretical findings on different benchmark problems, and assess the generalization capabilities on a series of challenging applications. Finally, we draw our conclusions and provide some insights on possible future developments.

\section{Geometrically parametrized PDEs: formulation \& properties}
\label{sec:formulation}
Within this work, we consider PDEs parametrized by a set of physical parameters, namely $\muv \in \mathcal{P} \subset \mathbb{R}^p$, and a set of geometrical parameters, denoted by $\xiv \in \mathcal{G} \subset \mathbb{R}^g$ for some $p,g \ge 0$:
\begin{equation}
    \label{eq:general_formulation}
    \left\{
    \begin{array}{rll}
        \partial_t u + \mathcal{N}[\muv,\xiv](u) &= 0, \qquad & \mbox{in} \ \Omega(\xiv) \times (0,T] \\
        \mathcal{B}[\muv,\xiv](u) &= 0,  \qquad & \mbox{on} \ \partial{\Omega}(\xiv) \times (0,T] \\
        u(0, \muv, \xiv) - u_0(\muv, \xiv) &= 0, \qquad& \mbox{in} \  \Omega(\xiv),
    \end{array}
    \right. 
\end{equation}
where $T > 0$; here $\mathcal{N}$ is a generic nonlinear operator,  $\mathcal{B}$ enforces the boundary conditions, and $u_0(\muv, \xiv)$ denotes the initial data. For the sake of the well-posedness of the problem, hereon we consider $\mathcal{P}$ and $\mathcal{G}$ to be compact subsets of $\mathbb{R}^p$ and $\mathbb{R}^g$, respectively. Then, we define $\mathcal{T} = [0,T]$ as the time domain. We emphasize that the domain shape depends on a set of parameters, that is, $\Omega = \Omega(\xiv) \subset \mathbb{R}^d$, for $d > 0$: we refer the reader to Fig. \ref{fig:diffeomorphism} for a visualization of an example of parametric domains. Note that the domain shape, while varying by changing $\xiv$, is fixed in time. Then, we highlight that the solution field $u =u(\bx; t, \muv, \xiv) \in \mathbb{R}$ depends on both physical and geometrical parameters, as well as on time, for any $\bx \in \Omega(\xiv)$. We also remark that, for the sake of simplicity, in this section we assume that the PDE solution is a scalar field, although the generalization of our results to vector problems is straightforward -- note that we will show numerical results related with Navier-Stokes equations at the end of the paper.

\begin{figure}[htb!]
    \centering
    \includegraphics[width=0.65\textwidth]{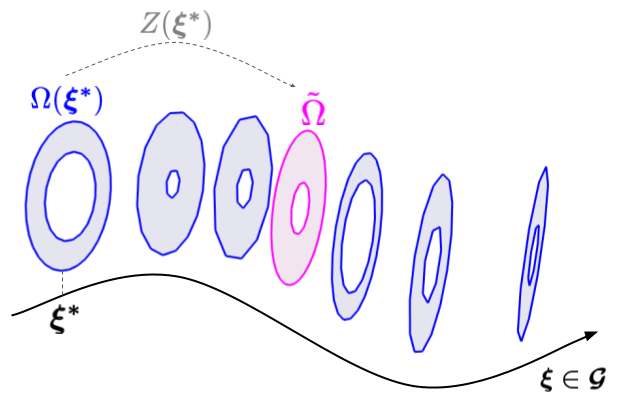}
    \caption{Visualization of a set of parametric domains obtained through a diffeomorphism $Z(\cdot)$; here the geometrical parameter $\xi$ regulates the radius of the circular hole included in each domain.}
    \label{fig:diffeomorphism}
\end{figure}

\subsection{Abstract setting}
We propose in this section an abstract setting for the analysis of the theoretical properties and the well-posedness of problem \eqref{eq:general_formulation}. The main difficulty of the aforementioned analysis originates from the fact that, for any instance $\xiv \in \mathcal{G}$ of the geometrical parameter, $u(t,\muv,\xiv)$ refers to different (parametrized) domains $\Omega(\xiv)$. In this respect, in the following we show that it is possible to perform a change of variables to cast the differential problem entailed by \eqref{eq:general_formulation} in a pre-defined reference configuration $\Omegaref$ that does not depend on the geometrical parameter $\xiv$. To do that, we introduce a special mapping that characterizes the deformation from the original configuration $\Omega(\xiv)$ onto the reference domain $\Omegaref$ (and vice versa), stating the following assumption.

\begin{assumption}
\label{assumption:diffeomorphism}
    Let us denote by $\tilde{\xiv} \in \mathcal{G}$ a reference value for the geometrical parameters, and define $\Omegaref = \Omega(\tilde{\xiv})$ as the reference domain. Then, there exists $(\bx,\xiv) \mapsto Z(\bx;\xiv)$ which is a $C^r$-diffeomorphism ($r\ge1$) as function of $\bx \in \mathbb{R}^d$ and Lipschitz as function of $\xiv \in \mathcal{G}$, such that:
    \begin{equation*}
        \Omega(\xiv) = \{\bx = Z^{-1}(\tilde{\bx};\xiv) \in \mathbb{R}^d \textnormal{, for any } \tilde{\bx}\in \Omegaref\}, \qquad \forall \xiv \in \mathcal{G}.
    \end{equation*}
    We also denote by $\zeta(\bx; \xiv) := |\nabla_{\bx} Z(\bx;\xiv)|$ the Jacobian of the transformation $Z$.
\end{assumption}

By stating Assumption \ref{assumption:diffeomorphism} we do not only aim at providing a definition of the mapping $Z$, but we also characterize its regularity. To this end, we stress that requiring $(\cdot,\xiv) \mapsto Z(\cdot, \xiv)$ to be at least homeomorphic, for any $\xiv \in \mathcal{G}$, ensures that the deformation from the original configuration $\Omega(\xiv)$ to the reference domain $\Omegaref$ preserves its topological invariants such as connectedness and compactness. In this way, it is straightforward to conclude that all the parametric domains share the same topological properties. However, it is not enough to require that $(\cdot,\xiv) \mapsto Z(\cdot, \xiv)$ is homeomorphic. Indeed, we must require $Z(\cdot,\xiv)$ and its inverse to be sufficiently smooth in order to cast \eqref{eq:general_formulation} in the reference configuration by performing the change of variable $\bx = Z^{-1}(\tilde{\bx};\xiv)$, namely, 
\begin{equation}
\label{eq:general_formulation_reference}
    \left\{
    \begin{array}{rll}
        \partial_t \tilde{u} + \mathcal{\tilde{N}}[\muv,\xiv](\tilde{u}) &= 0, \qquad &\mbox{in} \  \Omegaref \times (0,T] \\
        \mathcal{\tilde{B}}[\muv,\xiv](\tilde{u}) &= 0,  \qquad & \mbox{on} \ \partial{\Omegaref} \times (0,T] \\
        \tilde{u}(0, \muv, \xiv) - \tilde{u}_0(\muv, \xiv) &= 0, \qquad& \mbox{in} \  \Omegaref,
    \end{array}
    \right. 
\end{equation}
where $\tilde{u}(\tilde{\bx}; t, \muv, \xiv) = u(Z^{-1}(\tilde{\bx}; \xiv); t, \muv, \xiv)$ for any $\bx = Z^{-1}(\tilde{\bx}; \xiv) \in \Omega(\xiv)$ and $\mathcal{\tilde{N}}$ and $\mathcal{\tilde{B}}$ are the counterparts of the nonlinear operator $\mathcal{N}$ and the linear operator $\mathcal{B}$ on the reference domain, namely $\tilde{\mathcal{N}}[\tilde{\bx};\muv, \xiv] = \mathcal{N}[Z^{-1}(\tilde{\bx};\xiv);\muv, \xiv]$ and $\tilde{\mathcal{B}}[\tilde{\bx};\muv, \xiv] = \mathcal{B}[Z^{-1}(\tilde{\bx};\xiv);\muv, \xiv]$. We stress that the smoothness coefficient $r\ge1$ is related to the highest derivative degree entailed by $\mathcal{N}$ and $\mathcal{B}$. It is now evident that problem \eqref{eq:general_formulation_reference} is no longer geometrically parametrized, since now $\xiv \in \mathcal{G}$ plays the role of a mere physical parameter. For this reason, it is possible to proceed through classical arguments to establish the well-posedness of problem \eqref{eq:general_formulation_reference}: we refer the interested reader to \cite{quarteroni2016redb} and to Appendix \ref{sec:appendix-well-posedness} for further insights on the formal proof. As a final remark, even though within the present section we only considered the strong formulation of a PDE problem, a similar analysis can be carried out by employing the variational formulation as well, by performing a suitable change of variables in the integral form of the problem. Theoretical tools introduced so far allow us to adequately characterize geometrical variability in differential problems: not only, they will be essential in the following sections to describe the synthetic data generation phase, and to provide further insights on how geometrical variability might impact on the dimensionality reduction task.

\subsection{Synthetic data generation}
Within the present section, we aim at characterizing the training/testing dataset generation in the context of problems featuring geometrical variability. Since the solution of problem \eqref{eq:general_formulation} is not available in exact form, we harness a high-fidelity solver to obtain a suitably accurate numerical approximation. Specifically, in the following we distinguish two different strategies to synthetically generate data samples.

We emphasize that if the pair $(Z,\Omegaref)$ is known and fixed a priori, it is possible to systematically generate the realizations of the map $\xiv \mapsto \Omega(\xiv)$. In the literature, several approaches were proposed to suitably define $Z$; among them, we mention Free-Form Deformations (FFD) \cite{manzoni2012reduced,manzoni2012shape,lassila2010parametric}, Radial Basis Function (RBF) interpolation \cite{manzoni2012reduced,siena2023data,casenave2023mmgp} or mesh motion techniques \cite{manzoni2017}. When dealing with these strategies, it is in general possible to generate high-fidelity solutions by deforming (through $Z$) an adequately rich set of high-fidelity basis functions defined on the reference domain. Formally, being $W$ a suitable Hilbert space, we characterize the \textit{fixed resolution} approach through the following Assumption.

\begin{assumption}
\label{assumption:deformable-density}
     {\em (Fixed resolution approach)}. We suppose that, for any $\varepsilon > 0$, there exist $N_h$ and a finite dimensional subspace $W_h\subset W$ of dimension $N_h$, $W = \spn\{\tilde{\phi}_i\}_{i=1}^{N_h}$, with $\phi^*_i \in L^\infty(\Omegaref)$ for $i=1,\ldots,N_h$, such that
    \begin{equation*}
        \sup_{(t, \muv, \xiv) \in \mathcal{T} \times \mathcal{P} \times \mathcal{G}} \|u(t, \muv, \xiv) - u_h(t, \muv, \xiv)\|_* < \varepsilon,
    \end{equation*}
    where $\|\cdot\|_*$ is a suitable norm (depending upon the problem) and 
    \[
    u_h(\bx;t, \muv, \xiv) = \sum_{i=1}^{N_h} u_{h,i}(t, \muv, \xiv)\tilde{\phi}^*_i(Z(\bx,\xiv)) \qquad \mbox{for any } t \in \mathcal{T}, \muv \in \mathcal{P}, \xiv \in \mathcal{G} \ \ \mbox{and} \ \bx \in \Omega(\xiv).
    \]
\end{assumption}

\begin{figure}[b!]
    \centering
\includegraphics[width=0.85\textwidth]{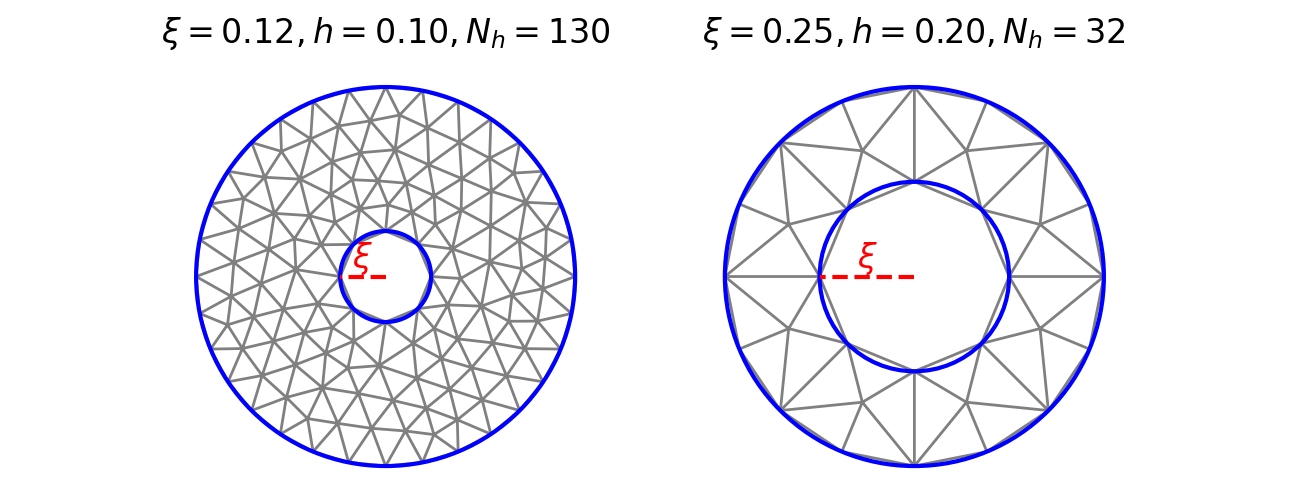}
    \caption{Example of parametric domains with different resolutions (the geometrical parameter $\xi$ is the radius of the hole).}
    \label{fig:example-different-dofs}
\end{figure}

We emphasize that Assumption \ref{assumption:deformable-density} relies on the assumption that the pair $(Z,\Omegaref)$ is known explicitly, which is not always the case in real applications.
Moreover, it is worth remarking that, within this approach, the number of dofs $N_h$ is the same for all data samples: for this reason, it is quite common in the ROM literature, since the majority of ROM techniques rely on a space-discrete formulation. However, while certain domain configurations may require fine discretizations in order to capture possible low-scale effects, other domain instances might only require a coarser discretization to properly grasp the underlying physics. Thus, in order to limit the dataset size by avoiding the storage of redundant information, we may choose to focus on \textit{multi-resolution} datasets, namely,
\begin{assumption}
\label{assumption:adhoc-discretizations-density}
      {\em (Multi-resolution approach)}. We suppose that, for any $\varepsilon > 0$ and any $\xiv \in \mathcal{G}$, there exist $N_h = N_h(\xiv)$ and a finite dimensional subspace $W_h \subset W$ of dimension $N_h$, $W_h = \spn\{\phi_i(\xiv)\}_{i=1}^{N_h}$, with $\phi_i(\xiv) \in L^\infty(\Omega(\xiv))$ for $i=1,\ldots,N_h$, such that
    \begin{equation*}
        \sup_{(\muv,t) \in \mathcal{T} \times \mathcal{P}} \|u(t, \muv, \xiv) - u_h(t, \muv, \xiv)\|_* < \varepsilon,
    \end{equation*}
    where $\|\cdot\|_*$ is a suitable norm (depending upon the problem) and 
    \[u_h(\bx;t, \muv, \xiv) = \sum_{i=1}^{N_h} u_{h,i}(t, \muv, \xiv)\phi_i(\bx,\xiv) \qquad \mbox{for any } t \in \mathcal{T}, \muv \in \mathcal{P}, \xiv \in \mathcal{G} \ \ \mbox{and} \ \bx \in \Omega(\xiv).
    \]
\end{assumption}
Hence, differently from the \textit{fixed resolution} approach, the \textit{multi-resolution} synthetic data generation does not require the availability of $(Z,\Omegaref)$ in a closed form. Anyway, with either a \textit{fixed resolution} or a \textit{multi-resolution} approach, we obtain an high-fidelity formulation of the semidiscretized problem that reads as follows:
\begin{equation}
    \label{eq:general_formulation_fom}
    \left\{
    \begin{array}{rll}
        \partial_t \mathbf{u}_h + \mathbf{N}[\muv,\xiv](\mathbf{u}_h) &= 0, \qquad & t \in (0,T] \\
        \mathbf{u}(0, \muv, \xiv) - \mathbf{u}_0(\muv, \xiv) &= 0 ,
    \end{array}
    \right. 
\end{equation}
where $\mathbf{u}_h := \mathbf{u}_h(t, \muv, \xiv) \in \mathbb{R}^{N_h}$ denotes the solution vector, $\mathbf{N}: \mathbb{R}^{N_h} \rightarrow \mathbb{R}^{N_h}$ is a generic nonlinear mapping that encodes the contribution of both $\mathcal{N}$ and $\mathcal{B}$ at a discrete level, whereas $\mathbf{u}_0(\muv, \xiv)$ is the parameter-dependent initial condition. We emphasize that, subject to the choice of the discretization strategy, $N_h$ might possibly depend on $\xiv$. Since the discretization is in principle parameter dependent (see Fig. \ref{fig:example-different-dofs}), in the following we cannot harness a finite-dimensional formulation. Therefore, the entire content of this paper is based upon an infinite-dimensional setting that we need to establish, allowing us to handle generic datasets obtained through \textit{multi-resolution}-based synthetic data generation. 

\subsection{An infinite-dimensional setting for dimensionality reduction}
\label{subsec:dim-red-setting}
In this section, we first define a suitable infinite-dimensional functional setting and then characterize the solution manifold of the parametric problem defined in Eq. \eqref{eq:general_formulation}. We stress that the following results are relative to the $L^2$ setting, however they can be extended to other separable Hilbert spaces in a straightforward way. We emphasize that characterizing the functional setting of geometrically parametrized problems demands a special attention because the solutions are defined on different, parametrized domains $\Omega = \Omega(\xiv)$. In the following, we develop a strategy to address this problem by constructing a suitably weighted Hilbert space. We proceed as follows: we first describe the chosen abstract setting, and then we justify our choices in light of the theory.
In this respect, for any $\xiv \in \mathcal{G}$, we define
\begin{equation*}
    L_{\zeta}^2(\Omega(\xiv)) = \bigg\{ f : \Omega(\xiv) \rightarrow \mathbb{R} \ \ : \ \  \int_{\Omega(\xiv)} f(\bx)^2 \zeta(\bx,\xiv) d\bx < +\infty \bigg\},
\end{equation*}
identifying the scalar product as $(\cdot,\cdot)_{L_\zeta^2(\Omega(\xiv))}$ and its induced norm as $\|\cdot\|_{L_\zeta^2(\Omega(\xiv))}$. 
Thus, we aim at delineating a proper solution manifold as a bounded subset of a single separable Hilbert space. However, it is evident that the collection
\begin{equation*}
    \mathcal{S} = \{u(t, \muv, \xiv) \in L^2_\zeta(\Omega(\xiv)) \ : \  \muv \in \mathcal{P}, t \in \mathcal{T}, \xiv \in \mathcal{G}\}
\end{equation*}
is an ill-defined solution manifold in a conventional sense, since each solution $u(t, \muv, \xiv)$ is characterized on a different, parameter-dependent, Hilbert space, namely $L_\zeta^2(\Omega(\xiv))$.
However, motivated by the fact that $L_\zeta^2(\Omega(\xiv))$ is weighted by $\zeta$, which is the Jacobian of the diffeomorphism $Z$, it is possible to show that a function of $L_\zeta^2(\Omega(\xiv))$ has a representative in $L^2(\Omegaref)$.
Indeed, we observe that, since $Z(\cdot,\xiv)$ is a $C^r$-diffeomorphism ($r \ge 1$), the following change of variable formula holds for any $f \in L_\zeta^2(\Omega(\xiv))$:
\begin{equation}
\label{eq:change-of-variable}
    \int_{\Omega(\xiv)} f(\bx)^2 \zeta(\bx,\xiv) d\bx = \int_{\Omegaref} f(Z^{-1}(\tilde{\bx},\xiv))^2  d\tilde{\bx}.
\end{equation}
Thus, it is possible to prove that, for any $f,g \in L_\zeta^2(\Omega(\xiv))$, the following relationship between scalar products holds:
\begin{equation*}
\begin{aligned}
(f,g)_{L^2_\zeta(\Omega(\xiv))} &= \int_{\Omega(\xiv)} f(\bx)g(\bx) \zeta(\bx,\xiv) d\bx \\
&= \int_{\Omegaref} f(Z^{-1}(\tilde{\bx},\xiv))g(Z^{-1}(\tilde{\bx},\xiv)) d\tilde{\bx} = (f \circ Z^{-1}(\xiv),g \circ Z^{-1}(\xiv))_{L^2(\Omegaref)},
\end{aligned}
\end{equation*}
where $\tilde{f}(\tilde{\bx}, \xiv) = f(Z^{-1}(\tilde{\bx},\xiv))$ and $\tilde{g}(\tilde{\bx}, \xiv) = g(Z^{-1}(\tilde{\bx},\xiv))$, for any $\xiv \in \mathcal{G}$ and $\tilde{\bx} \in \Omegaref$. 
Thus, upon defining the (linear) morphing operators
\begin{equation*}
\begin{array}{rlccccl}
    &\mathcal{Z}_{\xiv} &: &L^2(\Omegaref) \hspace{0.3cm} &\rightarrow &L^2_\zeta(\Omega(\xiv)) &\mbox{ as } \tilde{f} \mapsto \tilde{f} \circ Z(\xiv) \\
    &\mathcal{Z}_{\xiv}^{-1} &: &L^2_\zeta(\Omega(\xiv)) &\rightarrow &L^2(\Omegaref) &\mbox{ as } f \mapsto f \circ Z^{-1}(\xiv),
\end{array}
\end{equation*} 
the latter equality allows us to properly delineate the solution manifold of the parametric problem defined in Eq. \eqref{eq:general_formulation} in a conventional manner, namely, to define
\begin{equation*}
\begin{aligned}
    \tilde{\mathcal{S}} := \{\mathcal{Z}_{\xiv}^{-1}(u(t, \muv, \xiv)) \in L^2(\Omegaref) \ : \  \muv \in \mathcal{P}, t \in \mathcal{T}, \xiv \in \mathcal{G}\} \subset L^2(\Omegaref).
\end{aligned}
\end{equation*}

\begin{figure}
     \centering
     \begin{subfigure}[b]{0.99\textwidth}
         \centering
         \includegraphics[width=0.99\textwidth]{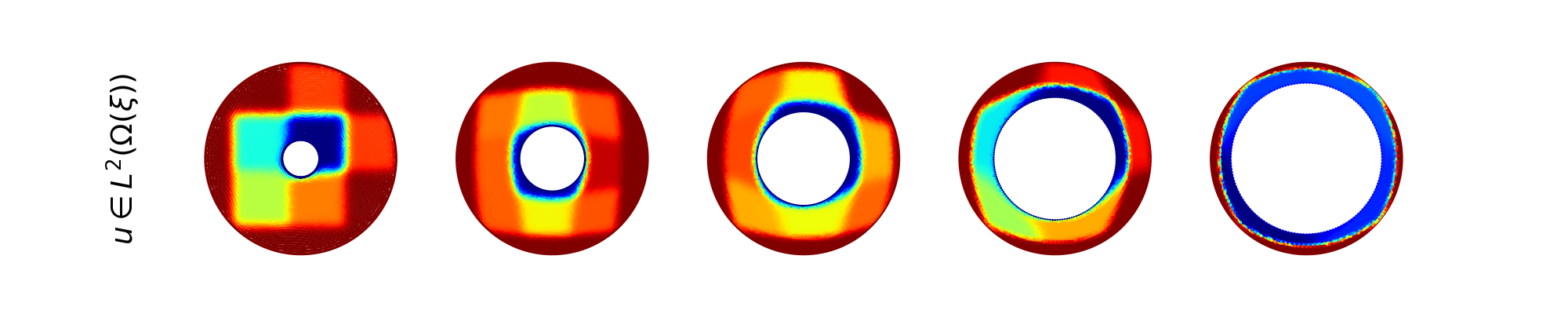}
     \end{subfigure} \\
     \begin{subfigure}[b]{0.99\textwidth}
         \centering
         \includegraphics[width=0.99\textwidth]{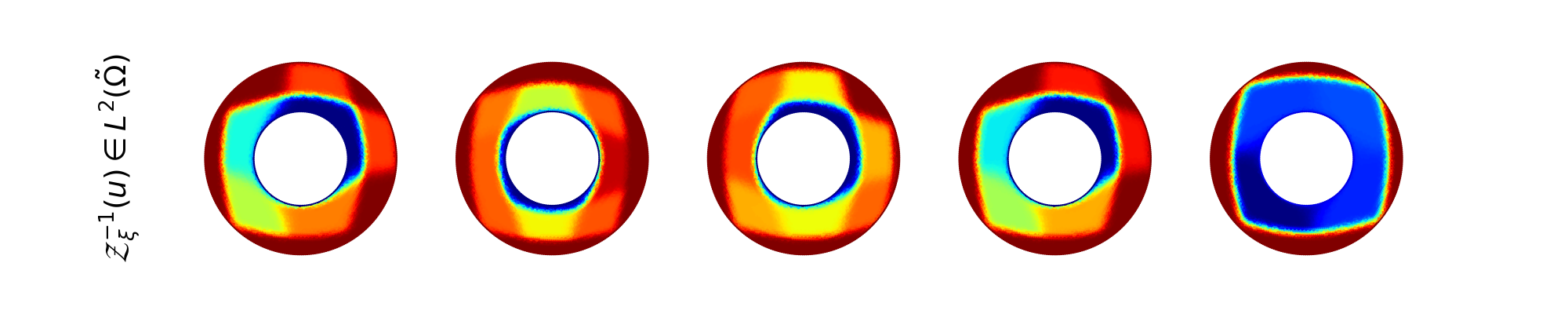}
     \end{subfigure} 
    \caption{Solution of a physically and geometrically parametrized advection diffusion reaction equation on parametric domains { \em(first row)} and action of the morphing operator $\mathcal{Z}_{\xiv}^{-1}$ {\em(second row)}.}
        \label{fig:morphing-operator}
\end{figure}

We refer the reader to Fig. \ref{fig:morphing-operator} for a visualization of the action of the operator $\mathcal{Z}_{\xiv}^{-1}$ and to Appendix \ref{sec:appendix-morphing} for a more technical analysis of $\mathcal{Z}_{\xiv}$ and $\mathcal{Z}^{-1}_{\xiv}$, which are proven to be linear bounded operators. We also emphasize that $\mathcal{Z}_{\xiv}$ and its inverse enable us to morph back and forth from the original configuration $\Omega(\xiv)$ onto the reference setting $\tilde{\Omega}$, namely,
\begin{equation*}
    \mathcal{S} = \{\mathcal{Z}_{\xiv} (\tilde{u}) \ : \  \tilde{u} \in \tilde{\mathcal{S}}\}, \qquad \tilde{\mathcal{S}} = \{\mathcal{Z}^{-1}_{\xiv} (u) \ : \ u \in \mathcal{S}\}.
\end{equation*}
Hence, thanks to this latter observation, we can set a dimensionality reduction framework in the reference configuration, thus enabling us to address the problem from a classical point of view since $\tilde{\mathcal{S}}$ is a proper subset of $L^2(\Omegaref)$ (see \S \ref{sec:infinite-dimensional-setting} and \S \ref{sec:modeling-basis-functions} for more details on the matter). 

Owing to the present functional setting, now we aim at providing a general framework to construct DL-ROMs in infinite dimensions, specifically suited for problems featuring geometrical variability.

\section{CGA-DL-ROMs in an infinite-dimensional setting}
\label{sec:infinite-dimensional-setting}
Within this section, we provide the principles behind the design of the CGA-DL-ROM architecture. Specifically, we characterize its distinguishing traits, namely { \em (i) } the geometry-aware dimensionality reduction in an infinite-dimensional setting, { \em (ii) } the nonlinear autoencoder, and { \em (iii) } the reduced network. We refer to Fig. \ref{fig:CGA-DL-ROM} for a schematic representation of the full architecture.

\vspace{2mm}
\noindent \textbf{CGA projection and lifting.} 
The purpose of the outermost block of the architecture is to provide a finite-dimensional representation of the infinite-dimensional solution field through a linear projection. To do that, we employ the space-continuous projection operator $\mathcal{V}^\dagger_{\xiv}:   L_\zeta^2(\Omega(\xiv)) \rightarrow \mathbb{R}^N$ and the space-continuous lifting operator $\mathcal{V}_{\xiv}: \mathbb{R}^N \rightarrow  L_\zeta^2(\Omega(\xiv))$.

Owing to \S\ref{subsec:dim-red-setting}, it is possible to decompose the action of each of the operators $\mathcal{V}^\dagger_{\xiv}$ and $\mathcal{V}_{\xiv}$ into two separate steps, namely,
\begin{equation*}
\setstretch{2.5}
\begin{array}{rrccccc}
    &\mathcal{V}^\dagger_{\xiv} : &L_\zeta^2(\Omega(\xiv)) &\xrightarrow[morphing]{\mathcal{Z}^{-1}_{\xiv}} &L^2(\Omegaref) &\xrightarrow[projection]{} &\mathbb{R}^N, \\ 
    &\mathcal{V}_{\xiv}  : &\mathbb{R}^N &\xrightarrow[lifting]{}  &L^2(\Omegaref)  &\xrightarrow[morphing]{\mathcal{Z}_{\xiv}} &L_\zeta^2(\Omega(\xiv)).
\end{array}
\end{equation*}
Thus, it is evident that, in practice, the effective projection and lifting operations only concern the reference configuration $\Omegaref$. Thanks to this observation, we introduce the continuous geometry-aware (CGA) basis functions $\{v^{CGA}_n(\xiv)\}_{n=1}^N \subset L^2(\tilde{\Omega})$, a set of global basis functions that are defined on the reference configuration and that depend on $\xiv$, and provide a formal definition of
\begin{equation*}
\setstretch{1.75}
\begin{array}{rll}
    &\mathcal{V}^\dagger_{\xiv}(f) :=  \{(\mathcal{Z}^{-1}_{\xiv}(f), v^{CGA}_n(\xiv))_{L^2(\Omegaref)}\}_{n=1}^N  \qquad &\forall f \in  L_\zeta^2(\Omega(\xiv)), \\ 
    &\mathcal{V}_{\xiv} (\bm{a}) := \mathcal{Z}_{\xiv}\biggl(\sum_{n=1}^N a_n v^{CGA}_n(\xiv)\biggr) \qquad &\forall \bm{a} = \{a_n\}_{n=1}^N \in \mathbb{R}^N.
\end{array}
\end{equation*}
Then, owing to the change of variable formula of \S \ref{subsec:dim-red-setting} and to the linearity of the morphing operator $\mathcal{Z}_{\xiv}$, we can derive a more compact, alternative definition of space-continuous projection and lifting operators,
\begin{equation*}
\setstretch{1.75}
\begin{array}{rll}
    &\mathcal{V}^\dagger_{\xiv}(f) :=  \{(f, \mathcal{Z}_{\xiv}(v^{CGA}_n(\xiv)))_{L_\zeta^2(\Omega(\xiv))}\}_{n=1}^N \qquad &\forall f \in  L_\zeta^2(\Omega(\xiv)), \\ 
    &\mathcal{V}_{\xiv} (\bm{a}) := \sum_{n=1}^N a_n \mathcal{Z}_{\xiv}(v^{CGA}_n(\xiv)) \qquad &\forall \bm{a} = \{a_n\}_{n=1}^N \in \mathbb{R}^N.
\end{array}
\end{equation*}

We emphasize that the compressive capabilities of the couple $(\mathcal{V}_{\xiv}, \mathcal{V}^{\dagger}_{\xiv})$ strongly depends on how we characterize the dependence on the geometrical parameters of the global basis functions $\{v^{CGA}_n(\xiv)\}_{n=1}^N$: we further discuss about this latter aspect in \S \ref{subsec:compressive-capabilities}. Nonetheless, since the collection $\{\mathcal{Z}_{\xiv}(v^{CGA}_n(\xiv))\}_{n=1}^N$ is in general unavailable in a closed form, we remark that within this work we model it through neural networks of weights and biases $\theta^v$, namely,
\begin{equation*}
    (\Omega(\xiv), \xiv) \ni (\bx, \xiv) \mapsto \{\hat{v}^{CGA}_n(\bx, \xiv)\}_{n=1}^N \approx \{\mathcal{Z}_{\xiv}(v^{CGA}_n(\xiv))(\bx)\}_{n=1}^N,
\end{equation*}
and we refer the interested reader to \S \ref{subsec:modeling-nn} for a more comprehensive outlook on the matter.

\begin{figure}[b!]
    \centering
    \vspace{-0.25cm}
    \includegraphics[width=0.925\textwidth]{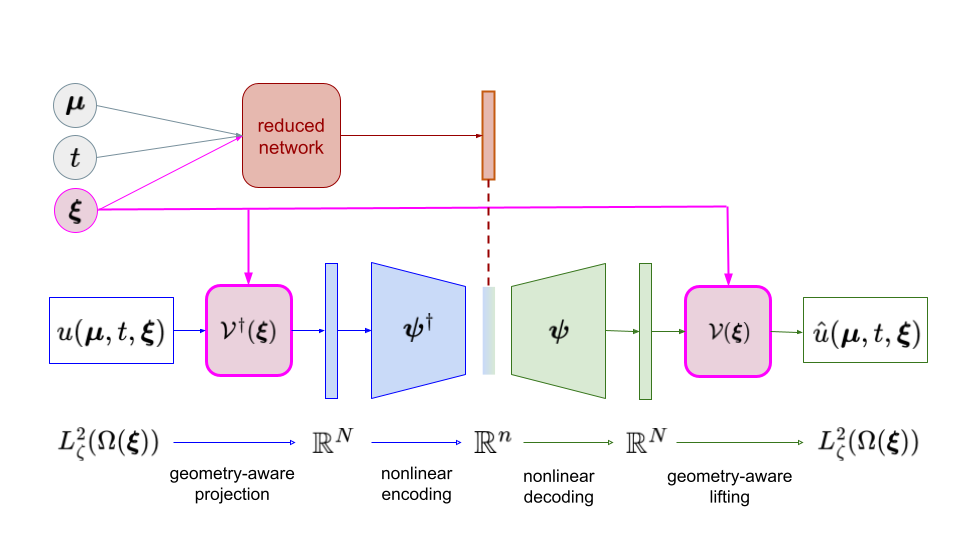}
    \caption{Schematic representation of the CGA-DL-ROM architecture.}
    \label{fig:CGA-DL-ROM}
\end{figure}

\vspace{2mm}
\noindent \textbf{Nonlinear autoencoder.} The purpose of the autoencoder architecture is to further compress the representation entailed by $(\mathcal{V}^\dagger_{\xiv}, \mathcal{V}_{\xiv})$ into a set of $l > 0$ latent coordinates. Thus, we configure the autoencoder architecture as the couple $(\psi^\dagger, \psi)$, where $\psi^\dagger: \mathbb{R}^N \rightarrow \mathbb{R}^l$ and $\psi: \mathbb{R}^l \rightarrow \mathbb{R}^N$ are such that $\psi^\dagger \circ \psi \approx Id$.  Normally, we parametrize the autoencoder $(\psi^\dagger, \psi)$ with (deep) neural networks and we denote with $\theta^{AE}$ the collection of weights and biases they depend on. We mention that, in the literature, both dense and convolutional architectures have proven to be effective for this purpose; see, e.g., \cite{fresca2021comprehensive,fresca2022poddlrom} for further details.

\vspace{2mm}
\noindent \textbf{Reduced network.} 
We design the reduced network as a feed-forward neural network architecture $\phi[\theta^\phi]: \mathbb{R}^{p+g+1} \rightarrow \mathbb{R}^{l}$, which is parametrized by its weights and biases $\theta^\phi$, and strives to approximate the reduced parametric map $(t, \muv, \xiv) \mapsto  \psi^\dagger \circ \mathcal{V}^\dagger (u(t, \muv, \xiv))$.

\vspace{2mm}
\noindent Thus, the resulting CGA-DL-ROM architecture is parametrized by $(\theta^{v},\theta^{AE},\theta^{\phi})$, which are the weights and biases of the geometry-aware basis functions, the autoencoder, and the reduced network, respectively. We emphasize that the optimization phase of CGA-DL-ROMs consists in seeking the optimal set of neural network weights (and biases) with respect to the following \textit{per-example} loss functional (for any $\muv \in \mathcal{P}, \xiv \in \mathcal{G}, t \in \mathcal{T}$), 
\begin{equation}
\label{eq:loss_continuous}
    \mathcal{L}_{CGA}(t, \muv, \xiv) = \|\mathcal{V}_{\xiv} \circ \psi \circ \phi(t, \muv, \xiv) -u(t, \muv, \xiv)\|^2_{L_\zeta^2(\Omega(\xiv))}  
    + \|\phi(t, \muv, \xiv) - \psi^\dagger \circ \mathcal{V}_{\xiv}^\dagger(u(t, \muv, \xiv))\|_2^2,
\end{equation}
which is composed of two different additive terms. The purpose of the former term is to control the squared reconstruction error; on the other hand, the latter term strives to constrain the architecture to find a suitable latent representation.  It is worth emphasizing that the CGA basis function are learnt in an unsupervised manner and that the encoder network is required only during the training phase. Indeed, similarly to other DL-ROM architectures, at inference stage we discard the encoder and use $\mathcal{V} \circ \psi \circ \phi$ for online computations. As a final remark, we also mention that in practice we are not able to compute exactly the integrals over $\Omega(\xiv)$, so that we replace \eqref{eq:loss_continuous} with its empirical counterpart. 

We finally stress that, while the nonlinear autoencoder and the reduced network have been analyzed in previous works \cite{brivio2024error, franco2023deep}, the impact of the novel proposed CGA projection and lifting is still to be analyzed.

\section{Characterization of CGA basis functions}
\label{sec:modeling-basis-functions}

We now propose an extensive analysis of CGA projection and lifting operators, providing also a characterization of their compressive capabilities. Specifically, our purpose is to show that only a small number $N$ of CGA basis functions is sufficient to capture the variability of the solution manifold (see \S \ref{subsec:compressive-capabilities}). This ensures a light neural network architecture and aims at enhancing the overall approximation capabilities. Moreover, since such basis functions are not analytically available, we provide  in \S \ref{subsec:modeling-nn} a suitable approximation through neural networks.

\subsection{The compressive capabilities of CGA basis functions}
\label{subsec:compressive-capabilities}
Before rigorously presenting the CGA approach, in order to fully appreciate the compressive capabilities of CGA basis functions, we first introduce the classical POD framework, which is the baseline we want to compare our approach with. First, we stress that, thanks to the results of the previous sections, we are allowed to cast the dimensionality reduction problem on the reference configuration $\tilde{\Omega}$. Thus, we characterize the POD approach in a rather general way, through a result formulated in a Hilbert context, relying on the theoretical setting reported in \cite{locke2020proper,singler2014pod} (we refer the interested reader to these latter works and the references therein for the proof).

\begin{lemma}
\label{lemma:pod-energy-continuous}
    Letting $u \in L^2(\mathcal{T} \times \mathcal{P} \times \mathcal{G}; L^2(\Omega(\xiv)))$, we define the POD cost functional as
    \begin{equation}
    \label{eq:POD-optimization-functional}
        \mathcal{J}_{POD}(\{w_n\}_{n=1}^N) = \bigg\|\mathcal{Z}_{\xiv}^{-1}(u) - \sum_{k=1}^{N} (\mathcal{Z}_{\xiv}^{-1}(u), w_n)_{L^2(\Omegaref)} w_n \bigg\|^2_{L^2(\mathcal{T} \times \mathcal{P} \times \mathcal{G};L^2(\Omegaref))},
    \end{equation}
    where $W_N = \{\{w_n\}_{n=1}^N \in L^2(\Omegaref) : (w_n,w_m)_{L^2(\tilde{\Omega})}=1, \quad \forall n,m \in \{1,\ldots,N\}\}$. Then, 
    \begin{equation}
    \label{eq:POD-optimization-problem}
    \begin{aligned}
        \mathcal{E}_{POD}(N) = \min_{\{w_n\}_{n=1}^N \in W_N} \mathcal{J}_{POD}(\{w_n\}_{n=1}^N)  &= \mathcal{J}_{POD}(\{v^{POD}_n\}_{n=1}^N)  = \sum_{n>N} l_n  < +\infty,
    \end{aligned}
    \end{equation}
    that is, over all possible linear finite-dimensional subspaces $W_N$, the projection error attains its minimum for the POD eigenpairs $(v^{POD}_n,l_n)_n \subset L^2(\Omegaref) \times \mathbb{R}$ collecting the eigenfunctions and the eigenvalues of the compact linear operator $K : L^2(\mathcal{T} \times \mathcal{P} \times \mathcal{G}) \rightarrow L^2(\Omegaref)$ defined as 
    \[
    Kg = (u(\bx), g)_{L^2(\mathcal{T} \times \mathcal{P} \times \mathcal{G})}, \qquad \mbox{for any } g \in L^2(\mathcal{T} \times \mathcal{P} \times \mathcal{G}).
    \]
\end{lemma}

In other words, the previous lemma states in a more general setting the well-known result for which POD provides the best set of basis functions \textit{on average} (in the $L^2$ sense) with respect to the physical and geometrical parameters, according to the projection error metric. In this respect, we stress that physical and geometrical parameters actually play the same role in the POD optimization functional \eqref{eq:POD-optimization-functional} but usually induce a different effect on the variability of the solution manifold. Indeed, a slow eigenvalue decay -- that is, when $l_N$ exhibits a polynomial decay as $N \rightarrow \infty$ -- is usually caused by the geometrical parametrization \cite{franco2024deep}, while the physical parameters often have a marginal contribution. The latter scenario might require the user to employ a large number $N$ of reduced basis functions to suitably reconstruct the variability of solution manifold, which turns to be extremely inefficient from a computational viewpoint. In this respect, the CGA approach aims at mitigating the aforementioned inefficiency by ultimately providing, \textit{for any instance of the geometrical parameter} $\xiv \in \mathcal{G}$, a set of modal basis functions that is optimal for the chosen parameter $\xiv$. This latter approach better reflects the different effects that physical and geometrical parameters have on the solution manifold variability. More formally, we characterize the CGA optimization problem through the following result.
\begin{lemma}
\label{lemma:CGA-functional}
     Letting $u \in L^2(\mathcal{T} \times \mathcal{P} \times \mathcal{G}; L^2(\Omega(\xiv)))$, we define the CGA cost functional as
    \begin{equation}
    \label{eq:CGA-functional}
    \mathcal{J}_{CGA}(\xiv, \{w_n\}_{n=1}^{N}) := \bigg\|\mathcal{Z}_{\xiv}^{-1}(u) - \sum_{n=1}^{N} (\mathcal{Z}_{\xiv}^{-1}(u), w_n)_{L^2(\Omegaref)} w_n \bigg\|^2_{L^2(\mathcal{T} \times \mathcal{P};L^2(\tilde{\Omega}))},
    \end{equation}
    and denote as $B = \{v \in L^2(\tilde{\Omega}) : \|v\|_{L^2(\tilde{\Omega})} = 1\}$ the unit ball in $L^2(\tilde{\Omega})$. Then, the CGA optimization problem
    \begin{equation}
    \label{eq:CGA-optimization-problem}
        \min_{\{w_n\}_{n=1}^{N} \in [B]^N}  \mathcal{J}_{CGA}(\xiv, \{w_n\}_{n=1}^{N}),
    \end{equation}
    is well-defined, that is, for any $\xiv \in \mathcal{G}$ there exists $\{v^{CGA}_n(\xiv)\}_{n=1}^N$ for which the CGA cost functional \eqref{eq:CGA-functional} attains the minimum. 
\end{lemma}

We refer the interested reader to Appendix \ref{sec:appendix-proofs} for the proof of the result above. Also, we summarize the POD and CGA optimization problems' properties in Table \ref{tab:optimization-problems}.
\bigskip
\begin{center}
    \renewcommand{\arraystretch}{1.6}
    \begin{tabular}{|c|c|c|}
        \hline
            & Optimization criterion & Basis functions\\
        \hline
        POD & average projection error on $L^2(\mathcal{T} \times \mathcal{P} \times \mathcal{G}; L^2(\Omegaref))$  & $\bx \mapsto \{v^{POD}_n(\bx)\}_{n=1}^N$ \\
        CGA & $\forall \xiv \in \mathcal{G}$, average projection error on $L^2(\mathcal{T} \times \mathcal{P}; L^2(\Omegaref))$ &  $(\bx,\xiv) \mapsto \{v^{CGA}_n(\bx,\xiv)\}_{n=1}^N$ \\
        \hline
    \end{tabular}
    \captionof{table}{Summary of the POD and CGA optimization problems.}
    \label{tab:optimization-problems}
\end{center}
\bigskip
Nonetheless, we highlight that, thanks to the functional setting provided in \S \ref{subsec:dim-red-setting}, it is straightforward to prove through the change of variable formula the following equalities, which show that the optimization problems defined in Eqs. \eqref{eq:POD-optimization-problem} and \eqref{eq:CGA-optimization-problem} can be equivalently cast in the original configuration $\Omega(\xiv)$, namely,
\begin{equation*}
\begin{array}{rll}
     &\mathcal{J}_{POD}(\{w_n\}_{n=1}^{N}) &= \bigg\|u -  \sum_{n=1}^{N} (u, \mathcal{Z}_{\xiv}(w_n))_{L^2(\Omegaref)} \mathcal{Z}_{\xiv}(w_n) \bigg\|^2_{L^2(\mathcal{T} \times \mathcal{P} \times \mathcal{G};L^2_{\zeta}(\Omega(\xiv))}\\
     &\mathcal{J}_{CGA}(\xiv, \{w_n\}_{n=1}^{N}) &= \bigg\|u -  \sum_{n=1}^{N} (u, \mathcal{Z}_{\xiv}(w_n))_{L^2(\Omegaref)} \mathcal{Z}_{\xiv}(w_n) \bigg\|^2_{L^2(\mathcal{T} \times \mathcal{P};L^2_{\zeta}(\Omega(\xiv))}.
\end{array}
\end{equation*}
Then, we can prove the main result of the present section, which showcases the fact that, in general, the CGA basis functions outperform POD modes, for the same dimension $N$ of the reduced approximation.
\begin{proposition}
    For any $N < \infty$, we define the POD and the CGA best approximation errors as
    \begin{equation*}
    \def\arraystretch{1.5}
    \begin{array}{rll}
        &\mathsf{BAE}_{CGA}(N) &:=  \int_{\mathcal{G}} \mathcal{J}_{CGA}(\xiv, \{v^{CGA}_n\}_{n=1}^{N}) d\xiv, \\
        &\mathsf{BAE}_{POD}(N) &:= \mathcal{J}_{POD}(\{v^{POD}_n\}_{n=1}^N),
    \end{array}
    \end{equation*}
    respectively.
    Then, we can prove that
    \begin{equation*}
        \mathsf{BAE}_{CGA}(N) \le \mathsf{BAE}_{POD}(N).
    \end{equation*}
\end{proposition}
\begin{proof}
We recall that $W_N = \{\{w_n\}_{n=1}^N \in L^2(\Omegaref) : (w_n,w_m)_{L^2(\tilde{\Omega})} = 1, \forall n,m \in \{1,\ldots,N\}\}$ and that $B$ is the unit ball in $L^2(\tilde{\Omega})$. Then, it is easy to see that $W_N \subset [B]^N$, so that $v_{POD} \in [B]^N$.
Thus, by an optimality argument, for any $\xiv \in \mathcal{G}$ it is valid that
\begin{equation*}
    \mathcal{J}_{CGA}(\xiv, \{v^{CGA}_n\}_{n=1}^{N}) = \min_{\{w_n\}_{n=1}^{N} \in [B]^N}  \mathcal{J}_{CGA}(\xiv, \{w_n\}_{n=1}^{N}) \le \mathcal{J}_{CGA}(\xiv, \{v^{POD}_n\}_{n=1}^{N}).  
\end{equation*}
Then, by integrating all the sides over $\mathcal{G}$, we derive the inequality
\begin{equation*}
    \int_{\mathcal{G}} \mathcal{J}_{CGA}(\xiv, \{v^{CGA}_n\}_{n=1}^{N}) d\xiv  \le \int_{\mathcal{G}} \mathcal{J}_{CGA}(\xiv, \{v^{POD}_n\}_{n=1}^{N}) d\xiv.
\end{equation*}
Moreover, by definition of $\mathsf{BAE}_{CGA}(N)$, we can state that,
\begin{equation*}
    \mathsf{BAE}_{CGA}(N)\le \int_{\mathcal{G}} \mathcal{J}_{CGA}(\xiv, \{v^{POD}_n\}_{n=1}^{N}) d\xiv.
\end{equation*}
Finally, by using the definition of POD and CGA optimization problems (Eqs. \eqref{eq:POD-optimization-problem} and \eqref{eq:CGA-optimization-problem}), for any $\{w_n\}_{n=1}^N \in W_N \subset [B]^N$, it is valid that
\begin{equation*}
    \int_{\mathcal{G}} \mathcal{J}_{CGA}(\xiv, \{w_n\}_{n=1}^{N}) d\xiv = \mathcal{J}_{POD}(\{w_n\}_{n=1}^{N}).
\end{equation*}
Thus, we can conclude that
\begin{equation*}
     \mathsf{BAE}_{CGA}(N)\le \int_{\mathcal{G}} \mathcal{J}_{CGA}(\xiv, \{v^{POD}_n\}_{n=1}^{N}) d\xiv = \mathcal{J}_{POD}(\{v^{POD}_n\}_{n=1}^{N}) = \mathsf{BAE}_{POD}(N).
\end{equation*}
\end{proof}

The better quality of the CGA basis functions when compared to the POD modes becomes evident in the cases where the geometrical parametrization is the only cause of the slow (polynomial) POD eigenvalue decay \cite{franco2024deep}, which entails that $\mathcal{E}_{POD}(N) = O(N^{-\alpha})$ for some $\alpha > 0$. On the other hand, if we consider the CGA approach involving \textit{ad hoc} basis functions for each geometrical parameter instance, the latter scenario entails that $\mathcal{J}_{CGA}(\xiv, \{v^{CGA}_n\}_{n=1}^{N}) = O(e^{-\beta N})$ for some $\beta > 0$, independently of  $\xiv \in \mathcal{G}$. Then, by definition, $\mathcal{E}_{CGA}(N) = O(e^{-\beta N})$, so that we can conclude that a small amount of CGA basis functions is sufficient to capture the solution manifold variability in this scenario, where a large amount of POD modes are necessary to obtain the same level of compression.

\subsection{Modeling CGA basis functions with neural networks}
\label{subsec:modeling-nn}
Despite having proved better compressive capabilities when compared to POD modes, in practical applications ``optimal" CGA basis functions are not available in a closed form, that is, the minimum in \eqref{eq:CGA-optimization-problem} cannot be explicitly attained, and needs to be suitably approximated. In this respect, within the present section we aim at providing a way to surrogate such CGA basis functions with neural networks. Specifically, the main existence result consists in proving that it is possible to construct a suitable neural network $(\bx,\xiv) \mapsto \{\hat{v}_n^{CGA}(\bx,\xiv)\}_{n=1}^N \approx \{\mathcal{Z}_{\xiv}(v^{CGA}_n(\bx,\xiv))\}_{n=1}^N$ such that the associated CGA projection error is arbitrarily close to the minimum in \eqref{eq:CGA-optimization-problem}, up to a prescribed tolerance. More formally,
\begin{proposition}
\label{prop:nn-appx}
    We let $N < \infty$ and assume that $(\bx,\xiv) \mapsto \{v^{CGA}_n(\bx,\xiv)\}_{n=1}^N \in W^{1,2}(\Omegaref \times \mathcal{G})$. Then, there exists $\varepsilon^* > 0$ such that, for any $\varepsilon \in (0,\varepsilon^*)$, it is possible to design a neural network architecture $(\bx,\xiv) \mapsto \{\hat{v}^{CGA}_n(\bx,\xiv)\}_{n=1}^N \approx \{\mathcal{Z}_{\xiv}(v^{CGA}_n(\bx,\xiv))\}_{n=1}^N$ so that
    \begin{equation*}
        \int_{\mathcal{G}} \mathcal{J}_{CGA}(\xiv, \{\mathcal{Z}^{-1}_{\xiv}(\hat{v}^{CGA}_n(\xiv))\}_{n=1}^{N}) d\xiv < \mathsf{BAE}_{CGA}(N) + \varepsilon.
    \end{equation*}
\end{proposition}
\noindent The proof is contained in Appendix \ref{sec:appendix-proofs}. Also, we point out that within this work, for the sake of simplicity, we require $(\bx,\xiv) \mapsto \{v^{CGA}_n(\bx,\xiv)\}_{n=1}^N$ to be sufficiently regular and we refer the interested reader to  \cite{franco2024measurabilitycontinuityparametriclowrank} for a detailed and broad analysis concerning the regularity of parametric low-rank approximation. Nonetheless, we emphasize that the map 
\[
(\bx,\xiv) \mapsto \{\hat{v}_n^{CGA}(\bx,\xiv)\}_{n=1}^N  \approx \{\mathcal{Z}_{\xiv}(v^{CGA}_n(\bx,\xiv))\}_{n=1}^N
\]
approximates both the geometry-aware basis functions and the action of the morphing operator $\mathcal{Z}_{\xiv}$, so that we do not require the diffeomorphism to be explicitly known in practical applications.

As a final remark, we highlight that we focus solely on the analysis of $\mathcal{J}_{CGA}$ and, by extension, of $\mathsf{BAE}_{CGA}$. Indeed, we neglect the contribution of the approximation capabilities of the reduced network and the decoder, and the effectiveness of the representation entailed by the encoder, whose analyses are beyond the purpose of the present work; for further details on this latter aspect, we refer the reader to \cite{bhattacharya2021, brivio2024error, franco2023deep, lanthaler2022error}. Notwithstanding, it is possible to show that $\mathsf{BAE}_{CGA}$ is a lower bound for the approximation error, that is, through an optimality argument we obtain that
\begin{equation*}
\begin{aligned}
    \mathcal{E}_{appx} :&= \|u - \hat{u}\|_{L^2(\mathcal{T} \times \mathcal{P} \times \mathcal{G};L^2_{\zeta}(\Omega(\xiv))} = \bigg\|u - \mathcal{V}_{\xiv} \circ \psi \circ \phi(t, \muv, \xiv) \bigg\|_{L^2(\mathcal{T} \times \mathcal{P} \times \mathcal{G};L^2_{\zeta}(\Omega(\xiv)))} \\
    & = \bigg\|u -  \sum_{n=1}^{N} (u, \mathcal{Z}_{\xiv}(\hat{v}^{CGA}_n))_{L^2(\Omega(\xiv))} \mathcal{Z}_{\xiv}(\hat{v}^{CGA}_n) \bigg\|_{L^2(\mathcal{T} \times \mathcal{P} \times{G};L^2_{\zeta}(\Omega(\xiv)))} \\
    & = \bigg[\int_{\mathcal{G}} \mathcal{J}_{CGA}(\xiv, \{\hat{v}^{CGA}_n\}_{n=1}^{N})^2 d\xiv \bigg]^{1/2} \\
    & \ge  \bigg[\int_{\mathcal{G}} \min_{\{v^{CGA}_n\}_{n=1}^{N} \in [B]^N} \mathcal{J}_{CGA}(\xiv, \{v^{CGA}_n\}_{n=1}^{N})^2 d\xiv \bigg]^{1/2} = \mathsf{BAE}_{CGA}(N),
\end{aligned}
\end{equation*}
where we recall $\psi, \phi$ are the decoder and the reduced network, respectively.

\section{Numerical experiments}
\label{sec:numerical-experiments}
Within this section we direct our efforts toward the implementation of numerical experiments supporting our theoretical findings and showcase the potential of the CGA-DL-ROM technique. More precisely,

\begin{itemize}
    \item[(i)] we motivate the usage of CGA basis functions by demonstrating their superiority in terms of compressive capability when compared to POD;
    \item[(ii)] we compare the CGA-DL-ROM with suitable baselines present in the literature;
    \item[(iii)] we test the accuracy of CGA-DL-ROM when tackling applications stemming from more-involved industrial benchmarks and challenges.
\end{itemize}
We remark that the accuracy of the predictions is measured by means of two relative error metrics, namely,
\begin{equation*}
\begin{aligned}
     &\mathcal{E}_{R} = \sum_{j=1}^{N_s^{test}} \frac{  \biggl( \sum_{i=1}^{N_h(\xiv_j)} (u(x_i;\muv_j, t_j, \xiv_j) - \hat{u}(x_i; \muv_j, t_j, \xiv_j))^2 \biggr)^{1/2}}{  \biggl( \sum_{i=1}^{N_h(\xiv_j)} u(x_i;\muv_j, t_j, \xiv_j)^2 \biggr)^{1/2}}, \\
     &E = \Bigg( \frac{ \sum_{j=1}^{N_s^{test}} \sum_{i=1}^{N_h(\xiv_j)} (u(x_i;\muv_j, t_j, \xiv_j) - \hat{u}(x_i; \muv_j, t_j, \xiv_j))^2}{   \sum_{j=1}^{N_s^{test}}\sum_{i=1}^{N_h(\xiv_j)} u(x_i;\muv_j, t_j, \xiv_j)^2} \Bigg)^{1/2},
\end{aligned}
\end{equation*}
where $\hat{u}$ is the approximation of the ground truth $u$. We stress that $\mathcal{E}_R$ is the most widely used metric in the field literature, while $E$ is a metric which is consistent with POD/CGA optimization problems. For this reason the latter is only employed to test the superiority of CGA basis functions with respect to POD modes.
Unless otherwise specified, { \em(i)} our numerical experiments are performed on a NVIDIA A100 80GB GPU, { \em(ii)} we employ Adam as optimizer, and { \em(iii)} we use $80\%$ of the generated data for training and the rest is split in equal parts for validation and testing. The ROM dimensions, the properties of the NN architectures as well as the computational costs entailed by training/testing stages for the considered test cases are summarized in Table \ref{tab_alldetails}.

\begin{center}
\setlength\extrarowheight{3pt}
\begin{small}
\begin{table}
\begin{tabular}{ | m{4.5em} | m{7em}| m{6.5em}| m{7em}| m{6.5em}| m{7em}|  } 
\hline
 &  & Stenosis  & Hyper-elasticity  & NS obstacle  & Heat exchanger \\
 &  & (\S \ref{subsec:stenosis}) &  (\S \ref{subsec:elasticity}) & (\S \ref{subsec:ns-obstacle}) & (\S \ref{subsec:heat-exchanger})\\
\hline
\multirow{2}{4.5em}{ROM dimensions} & latent dim. $l$ & 4 & 20 & 25 & 17\\ 
& red. dim. $N$ & 4 & 30 & 60 & 70\\ 
\hline
\multirow{4}{4.5em}{Archite\-cture specifics} & $\{\hat{v}_n^{CGA}\}_{n=1}^N$ & $\textnormal{RDense}(120,10)$ & $\textnormal{RDense}(180,11)$ & $\textnormal{RDense}(220,10)$ & $\textnormal{RDense}(150,10)$\\ 
& encoder $\psi^\dagger$ & $\textnormal{Dense}(150,5)$ & $\textnormal{Dense}(150,5)$ & $\textnormal{Conv2d}(4)$ & $\textnormal{Dense}(50,5)$\\ 
& decoder $\psi$ & $\textnormal{Dense}(150,5)$ & $\textnormal{Dense}(150,5)$ & $\textnormal{Conv2dT}(4)$ & $\textnormal{Dense}(150,5)$\\ 
& red. network $\phi$ & $\textnormal{Dense}(50,5)$ & $\textnormal{Dense}(50,5)$ & $\textnormal{Dense}(100,5)$ & $\textnormal{Dense}(80,5)$\\ 
\hline
\multirow{4}{4.5em}{Computa\-tional \\ burden} &\# NN weights & $196k$ & $397k$ & $943k$ & $298k$ \\ 
& Training time & $0.3h$ & $0.8h$ & $3h$ & $4h$ \\ 
& Inference time (per instance) & $1ms$ & $0.5ms$ & $1ms$ & $5ms$ \\
\hline
\multirow{4}{4.5em}{Dataset specifics} &\# samples $N_s$ & $1000$ & $2000$ & $100$ & $100$ \\ 
&splitting & $80\%/10\%/10\%$ & $50\%/40\%/10\%$ & $81\%/9\%/10\%$ & $80\%/10\%/10\%$ \\ 
& \# time steps $N_t$ & $-$ & $-$ & $300$ & $60$ \\ 
& \# dofs $N_h$ & $7.8k$ & $972$ & $31k \div 34k$ & $100k\div110k$ \\
\hline
\end{tabular}
\captionof{table}{Hyperparameters of best results obtained with CGA-DL-ROM in every experiment. We denote with $\textnormal{Dense}(w,d)$ the dense network of $d$ layers with maximum width $w$ and with $\textnormal{RDense}(w,d)$ the same architecture with residual connections. On the other hand, $\textnormal{Conv2d}(d)$ ($\textnormal{Conv2dT}(d)$) indicates a convolutional block consisting of $d$ convolutional (transposed convolutional) layers and two dense layers at input and output to correctly match input and output dimensions. }
\label{tab_alldetails}
\end{table}
\end{small}
\end{center}

\subsection{Stationary flow in a stenotic channel}
\label{subsec:stenosis}
We first consider a benchmark test case to demonstrate the superior compressive and generalization capabilities of the proposed CGA-DL-ROM architecture when compared to its POD counterpart. In this respect, we immediately stress that if we replace CGA basis functions with POD modes, we obtain the POD-DL-ROM architecture \cite{fresca2022poddlrom}, which indeed is the baseline we would like to compare our novel approach with.

The present numerical experiment involves a fluid flowing in a 2D blood vessel over a stenosis and is usually employed as an extremely simplified model for studying coronary artery diseases, which are of great interest in haemodynamics \cite{siena2023data}. Specifically, we revisit the numerical experiment \textit{6.4} of \cite{manzoni2012reduced}, namely,
\begin{equation}
    \left\{
    \begin{aligned}
        & - \nu \Delta \bm{u} + (\bm{u} \cdot \nabla) \bm{u} + \nabla p = \mathbf{0} ,& \mbox{in} \ \Omega(\xiv)\\
         & \nabla \cdot \bm{u} = 0,& \mbox{in} \ \Omega\\
        &\bm{u} = [Ay(1-y),0,0], & \mbox{on} \  \Gamma_{IN}\\
        &\bm{u} = \mathbf{0}, & \mbox{on} \  \Gamma_{W}(\xiv)\\
        &-p\bm{n} + \nu(\nabla\bm{u})\bm{n} = \mathbf{0}, & \mbox{on} \  \Gamma_{OUT},
    \end{aligned}
    \right.
\end{equation}
where $\Gamma_{IN}$ and $\Gamma_{OUT}$ are the inlet and outlet of the blood vessel, whereas $\Gamma_{W}(\xiv) = \partial\Omega(\xiv) \setminus (\Gamma_{IN} \cup \Gamma_{OUT})$. Moreover, we emphasize that the problem is physically parametrized by the fluid viscosity $\nu \in [0.002, 0.004]$ and the inlet magnitude $A \in [2,4]$.
We parametrize the geometry of the artery's lower wall by means of the function
\begin{equation*}
    w(y; \xiv) = 
    \left\{
    \begin{array}{ll}
        \xi_1 \cos\bigg(\frac{\pi}{2\xi_2}(y - \xi_3)\bigg), & x \in (\xi_3 - \xi_2, \xi_3 + \xi_2) \\
        0, & \textnormal{otherwise},
    \end{array}
    \right.
\end{equation*}
where $\xiv = (\xi_1, \xi_2, \xi_3) \in [0.25,0.4] \times [0.5, 0.75] \times [2, 3]$ collects the geometrical parameters which regulate the stenosis severity and position (see Fig. \ref{fig:stenosis-sample-mesh} for more details).

\begin{figure}[htb!]
    \centering
    \includegraphics[width=0.75\textwidth]{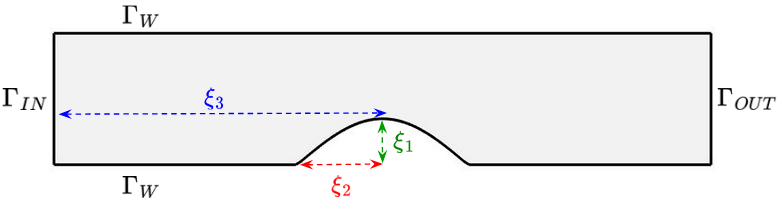}
    \caption{Stenosis test case: visualization of the effect of the geometrical parameters on the domain shape.}
    \label{fig:stenosis-sample-mesh}
\end{figure}

We generate a set of $N_s=1000$ samples collecting only the $x$-velocity and $y$-velocity components, which are the fields that we would like to reconstruct. To do that, we make use of \texttt{FEniCS} \cite{langtangen2017solving} and \texttt{gmsh} \cite{gauzaine2017gmsh}, employing a mixed $\mathbb{P}^1\textnormal{b}-\mathbb{P}^1$ finite element space and a Newton solver to deal with the nonlinearity. Generating all the samples takes $35$ minutes on a Intel Core i7 13th gen 32GB RAM laptop. 

We remark that, for the sake of simplicity, we make use of a \textit{fixed-resolution} approach and a RBF strategy to morph the meshes, thus entailing a total of $N_h=7600$ dofs for each sample. The latter assumption simplifies the comparison of CGA-DL-ROM with its POD counterpart because it is possible to pre-compute POD basis functions by means of SVD. 
Specifically, we train for $500$ epochs a series of CGA-DL-ROMs and POD-DL-ROMs to reconstruct both the velocity components of for varying reduced dimension $N \in \{2^m : m = 0, \ldots, 6\}$ and evaluate their accuracy on the test set. Also, we set the latent dimension $l = \min(N,10)$. From the analysis' results displayed in in Fig. \ref{fig:stenosis-analysis-N}, it is evident that CGA-DL-ROMs are more accurate than POD-DL-ROMs for any value of $N$: this behaviour is particularly evident for lower values of $N$, thus corroborating our theoretical derivations. We also point out that $E(N)$ reaches a plateau for both the analyzed paradigms. However, CGA-DL-ROM saturates earlier ($N \ge 4$) than POD-DL-ROM ($N \ge 16$) and the plateau value of CGA-DL-ROM is $3\times$ smaller than the one relative to the POD counterpart.

\begin{figure}[htb!]
    \hspace{2.3cm}
    \includegraphics[width=0.62\textwidth]{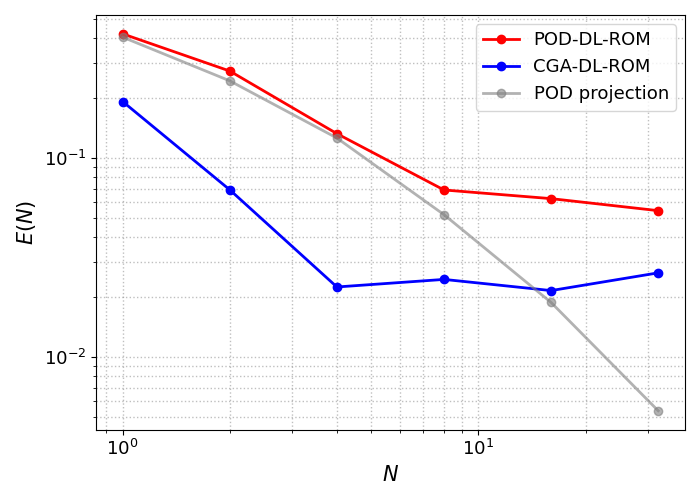}
    \caption{Stenosis test case: analysis of the compressive capabilities of CGA-DL-ROM and POD-DL-ROM. We also show the POD projection error for reference.}
    \label{fig:stenosis-analysis-N}
\end{figure}

\begin{figure}[htb!]
    \centering
    \includegraphics[width=0.99\textwidth]{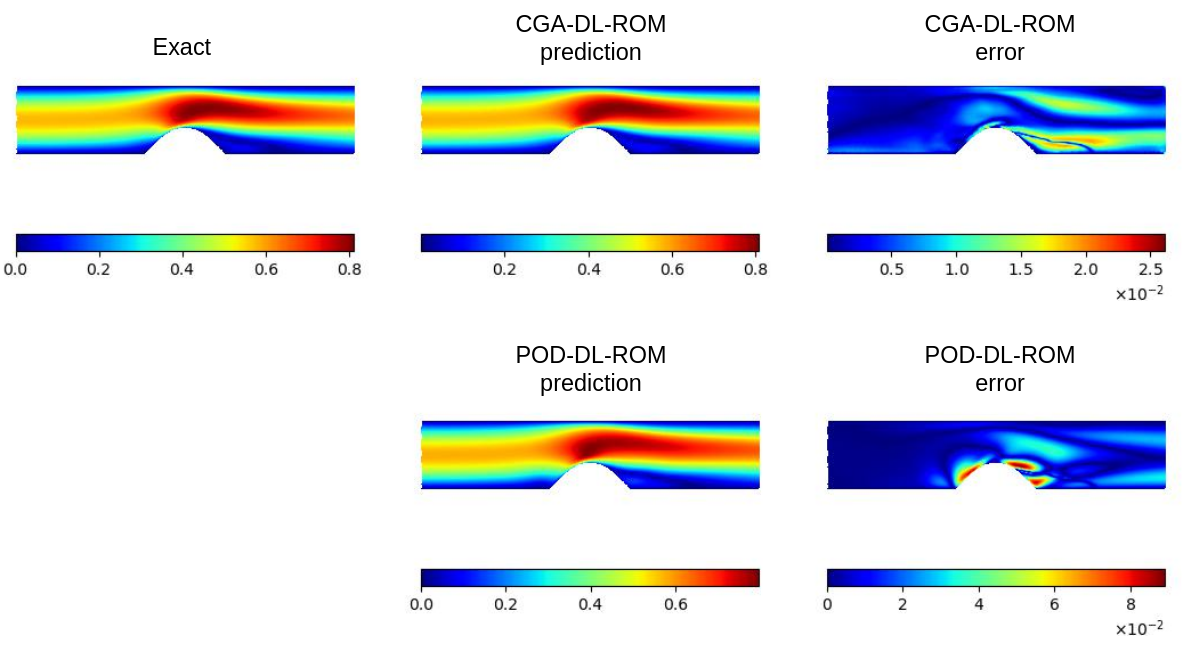}
    \caption{Stenosis test case: comparison of CGA-DL-ROM and POD-DL-ROM predictions and errors in the case of $N=4$.}
    \label{fig:comparison-CGA-POD}
\end{figure}

Finally, we emphasize that only a small number of CGA basis functions ($N=4$) is required to obtain a suitable accuracy. On the other hand, the same number of POD modes is not capable to suitably capture the solution manifold's variability. Indeed, POD-DL-ROM struggles especially in the proximity of the stenosis where the impact of the geometrical parameters is more evident (see Fig. \ref{fig:comparison-CGA-POD}).

\subsection{Hyper-elasticity equations}
\label{subsec:elasticity}
Within the present numerical experiment we aim at comparing the proposed approach with several baseline strategies present in the literature. To do that, we exploit the numerical experiment proposed in \cite{li2022fourier}, which is a common benchmark for geometrically parametrized problems. 
More precisely, the present numerical experiment involves an hyper-elastic material described by the solid body equation
\begin{equation*}
    \frac{\partial^2 \bm{u}}{\partial t^2} + \nabla \cdot \bm{\sigma} = \mathbf{0}
\end{equation*}
where $\bm{u}$ and $\bm{\sigma}$ are the displacement and the stress field, respectively, and the equation is endowed with a Rivlin-Saunders constitutive model. The body shape consists of a unit square with a parametric void at the center; the void is centered at the point $(0.5,0.5)$ and is characterized by the radius prior $r = 0.2 + \frac{0.2}{1 + \exp(r^*)}$ (where $r^* \sim \mathcal{N}(0,4^2(-\nabla + 3^3)^{-1})$, yielding a total of $42$ geometrical parameters per sample. The body is subject to the tensile traction $[0,100]$ imposed at the top edge $y=1$. Within the present test case we aim at reconstructing the stress given a geometry consisting of a point cloud and a geometrical parameter representing the void topology.

\begin{table}[ht!]
\centering
\setlength\extrarowheight{5pt}
\begin{tabular}{ | m{3cm} | P{2.5cm}| P{2cm} | P{3cm} | } 
\hline
& Relative error & Latent space dimension & \# NN weights and biases (in millions)\\
\hline
CGA-DL-ROM & $\underline{\mathit{1.26 \times 10^{-2}}}$ & $\underline{\mathbf{20}}$ & $\underline{\mathbf{0.40}}$ \\
Geo-FNO \cite{li2022fourier} & $2.20 \times 10^{-2}$ & 
$4608$ & $1.55$ \\
GraphNO \cite{li2022fourier} & $1.26 \times 10^{-1}$ & $-$ & $\underline{\mathit{0.57}}$\\ 
DeepONet \cite{li2022fourier} & $9.65 \times 10^{-2}$ & $\underline{\emph{256}}$ & $1.03$ \\
DAFNO \cite{liu2023domain} & $ \underline{1.09 \times 10^{-2}}$ & $4608$ & $2.37$\\
Coral \cite{serrano2023operator} & $1.64 \times 10^{-2} $ & $\underline{128}$ & $\underline{0.53}$ \\
F-FNO \cite{tran2023factorized} & $ 1.74 \times 10^{-2}$ & $4608$ & $8.87$ \\
GNOT \cite{hao2023gnot} & $\underline{\mathbf{8.65 \times 10^{-3}}}$ & $-$ & $> 2$\\
\hline
\end{tabular}
    \caption{Elasticity test case: comparison against baselines (less is better: \underline{\textbf{first}}, \underline{second}, \underline{\emph{third}}). The latent space dimension of FNO-based models is $l = \# \textnormal{modes}^{d} \times \textnormal{width}$, where $d$ is the number of spatial dimensions. The latent space of DeepONet is its number of basis functions. The reported results are taken from the associated paper and/or code.}
\label{table:elasticity-baselines}
\end{table}

\begin{figure}[h!]
    \centering
    \subfloat[Original domain]{
    \centering
    \includegraphics[width=0.99\linewidth]{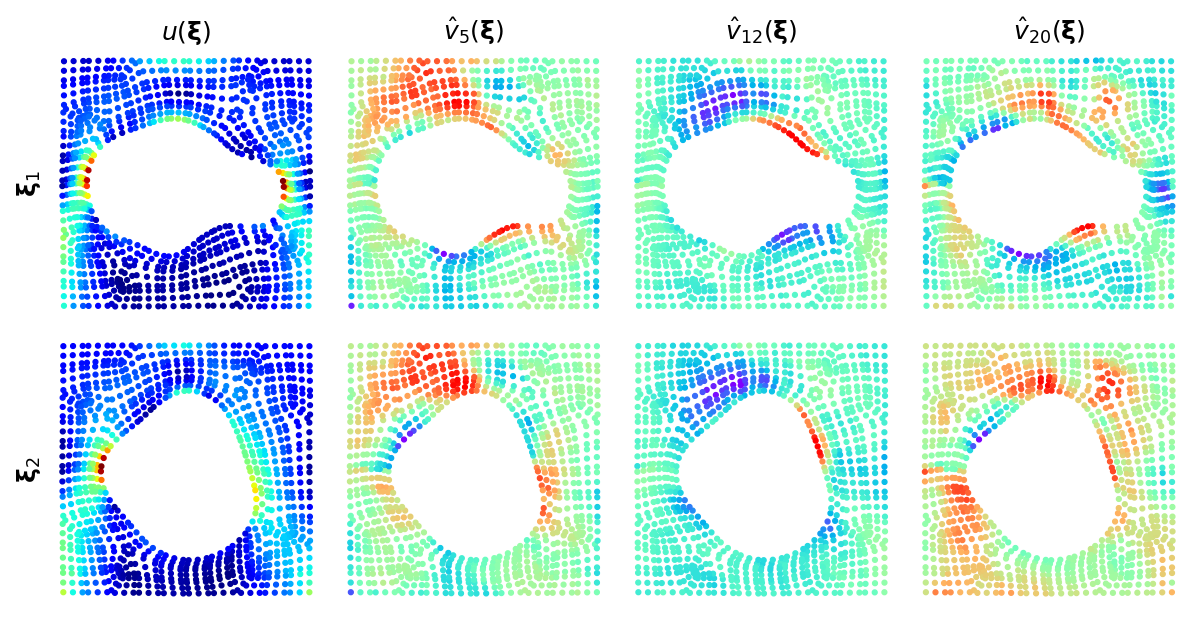}
    }
    \vspace{0.5cm}
    \subfloat[Reference domain]{
    \centering
    \includegraphics[width=0.99\linewidth]{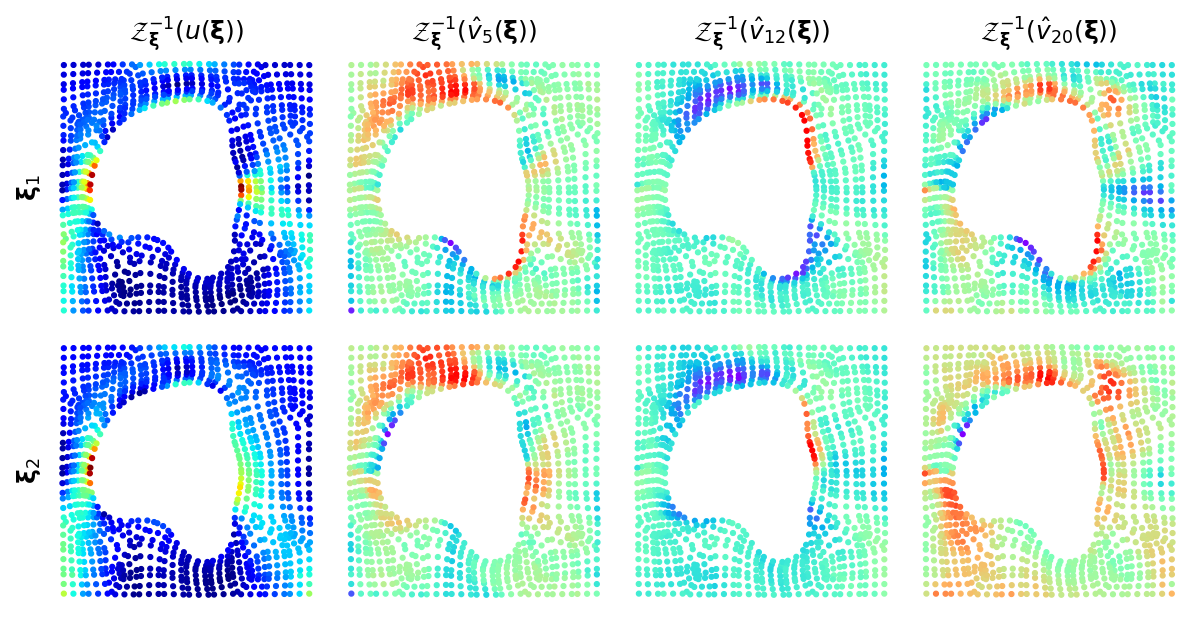}
    }
    \vspace{0.2cm}
    \caption{Hyper-elasticity test case: visualization of the ground truth {\em (first column)} and CGA basis functions {\em (other columns)} for two instances of the geometrical parameter {\bf (a)} on the original domain and {\bf (b)} morphed back on a reference domain.}
    \label{fig:basis-functions}
\end{figure}

\begin{figure}[htb!]
    \centering
    \includegraphics[width=0.94\textwidth]{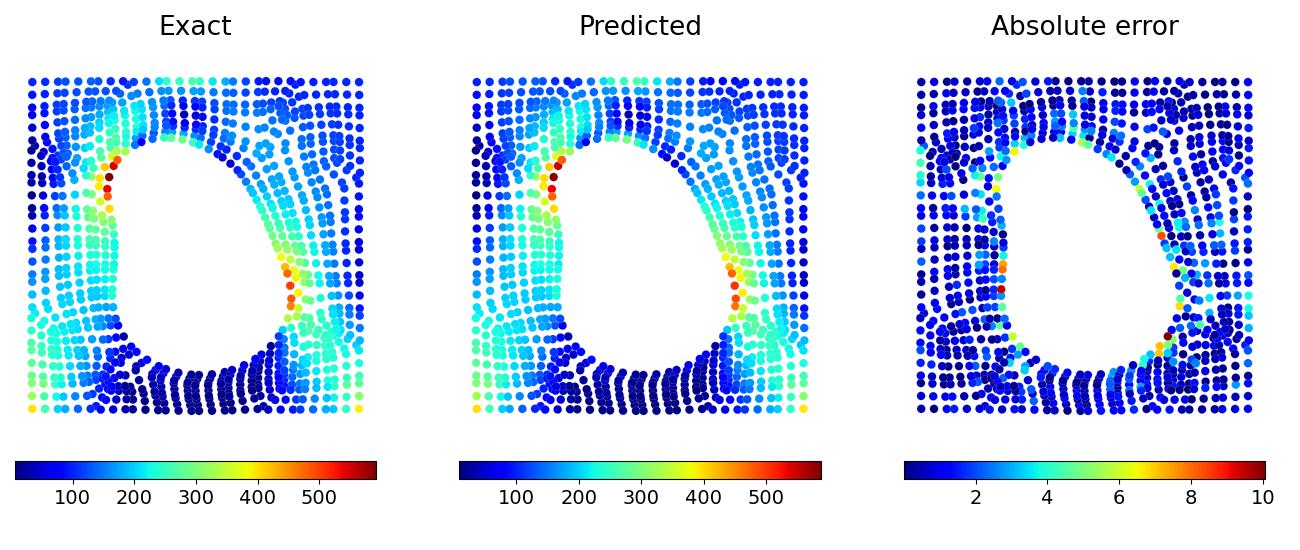}
    \caption{Hyper-elasticity test case: CGA-DL-ROM prediction and error for a test sample.}
    \label{fig:hyperelasticity_sample.png}
\end{figure}

We train CGA-DL-ROM on $1000$ data samples for $3000$ epochs with a batch size of $8$ and using a learning rate of $3 \times 10^{-4}$. 
Then, we compare our paradigm with the most widely used techniques in the literature, including
\begin{itemize}[itemsep=0.3pt]
    \item Neural operators based on learned deformations \cite{li2022fourier,tran2023factorized}, the first architectures specifically crafted for geometrically parametrized problems proposed in the literature;
    \item DAFNO \cite{liu2023domain}, which are based on a FNO core coupled with smoothed characteristic functions to describe the geometry, but are limited to uniform grids;
    \item GraphNO \cite{li2020multipole,li2022fourier} architectures, whose graph-based structure makes them suitable for dealing with parametrized geometries;
    \item DeepONets \cite{lu2021learning,li2022fourier}, which, in constrast to CGA-DL-ROMs, leverage on basis functions not depending on the geometrical parameter;
    \item Coral \cite{serrano2023operator}, a flexible mesh-free INR-based architecture showing promising results on several numerical experiments;
    \item GNOT \cite{hao2023gnot}, which leverages a specific attention mechanism to construct a strong inductive bias for a wide variety of operator learning tasks including (but not limited to) geometrically parametrized problems.
\end{itemize}

As shown in Table \ref{table:elasticity-baselines}, CGA-DL-ROMs are extremely efficient surrogate models. Indeed, they showcase an excellent accuracy on the test set (composed of 200 samples), almost on par with state-of-the-art paradigms, while having a smaller architecture (in terms of number of parameters) consisting of a remarkably low number of CGA basis functions ($N=30$) and an even lower latent dimension ($l=20$). On the other hand, transformer architectures (GNOT) deliver the best results in terms of accuracy, at the expense of heavy models. The only paradigm entailing a number of neural network parameters which is comparable with CGA-DL-ROM is Coral, but it features a worse accuracy and requires an inference-time optimization \cite{serrano2023operator}. We also note that modulating the basis functions by means of geometrical parameters is an effective strategy. Indeed, CGA-DL-ROMs are far more accurate (more than $7 \times$) than DeepONets, whose basis functions are independent of the geometrical parameter. From Fig \ref{fig:basis-functions}, we can qualitatively assess that CGA basis functions are deeply connected with the local features that the solution field exhibits in the proximity of the void, where the effect of the geometrical parameterization is particularly intense. Moreover, we further validate qualitatively the approximation and generalization capabilities of CGA-DL-ROM thanks to Fig. \ref{fig:hyperelasticity_sample.png}. Finally, we stress that, other than being suitably accurate and featuring a light architecture, the CGA-DL-ROMs are also versatile. Indeed, in contrast to \cite{li2022fourier, liu2023domain}, which often require practical extensions and data manipulation, the CGA-DL-ROM architecture can be used ``as is" to tackle a wide variety of geometrically parametrized problems, as shown in the next experiments.

\subsection{Navier-Stokes flow around an obstacle}
\label{subsec:ns-obstacle}
The present numerical experiment involves the reconstruction of the velocity field of a fluid flowing in a 2D rectangular channel $[0,3] \times [0,0.5]$ past a circular obstacle. The domain shape is geometrically parameterized by means of $\xiv = \{y_c, r_c\}$, where $y_c \in [0.15,0.35]$ and $r_c \in [0.05,0.1]$ are respectively the $y$-position of the center of the obstacle and the radius of the obstacle. The inlet boundary is $\Gamma_{IN} = \{(x,y) \in \partial{\Omega}(\xiv) : \bx = 0\}$, while the outlet boundary is $\Gamma_{OUT} = \{(x,y) \in \partial{\Omega}(\xiv) : x = 3\}$.
The problem formulation reads
\begin{equation}
    \left\{
    \begin{aligned}
        & \frac{\partial{\bm{u}}}{\partial t} - \nu \Delta \bm{u} + (\bm{u} \cdot \nabla) \bm{u} + \nabla p = \mathbf{0} ,& \mbox{in} \ \Omega(\xiv) \times (0,T]\\
        & \nabla \cdot \bm{u} = 0 ,& \mbox{in} \ \Omega(\xiv) \times (0,T]\\
        &\bm{u} = [h(y;A),0], & \mbox{on} \  \Gamma_{IN}(\xiv) \times (0,T]\\
        &\bm{u} = \mathbf{0}, & \mbox{on} \  \partial{\Omega(\xiv)} \setminus (\Gamma_{IN} \cup \Gamma_{OUT}) \times (0,T]\\
        &-p\bm{n} + \nu(\nabla\bm{u})\bm{n} = \mathbf{0}, & \mbox{on} \  \Gamma_{OUT}(\xiv) \times (0,T] \\
        & \bm{u}_0 = \bm{u}_{stokes} , & \mbox{in} \ \Omega(\xiv)\\
    \end{aligned}
    \right.
\end{equation}
where $T=0.8$ is the time horizon, $\bm{u}_{stokes}$ is the steady Stokes simulation on the same configuration, $\nu \in [1 \times 10^{-3}, 2 \times 10^{-3}]$ is the fluid viscosity and $h(y;A) = 4 A (0.5 - y) y$ is the inlet profile that is regulated by the maximum amplitude $A \in  [1.5^2,5\cdot1.5^2]$. For the sake of clarity, we represent the physical parameters with the vector $\muv = [A,\nu] \in \mathbb{R}^2$. We highlight that with the present choice of the physical parameters, since $Re \approx 550 \div 5625$, we witness vortex shedding phenomena for any of the sample instances, especially as $t$ approaches the time horizon. The high-fidelity simulations are obtained with \texttt{FEniCS} \cite{langtangen2017solving} and \texttt{gmsh} \cite{gauzaine2017gmsh} by discretizing the problem by means of $\mathbb{P}^2-\mathbb{P}^1$ couple of finite element spaces and by employing the Chorin-Temam splitting scheme, yielding a total of more than $3 \times 10^4$ dofs. For the time-advancing scheme, we chose $\Delta t = T / 3000$ to ensure the stability of the numerical solution, but we save the solution snapshots only each $10$ iterations in order to limit the memory footprint of the dataset, while retaining the essential features of the temporal evolution of the solution fields. The elapsed time to generate the high-fidelity dataset is $23h$ on a Intel Core i7 13th gen 32GB RAM laptop.

\begin{figure}[htb!]
    \centering
    \includegraphics[width=0.49\textwidth]{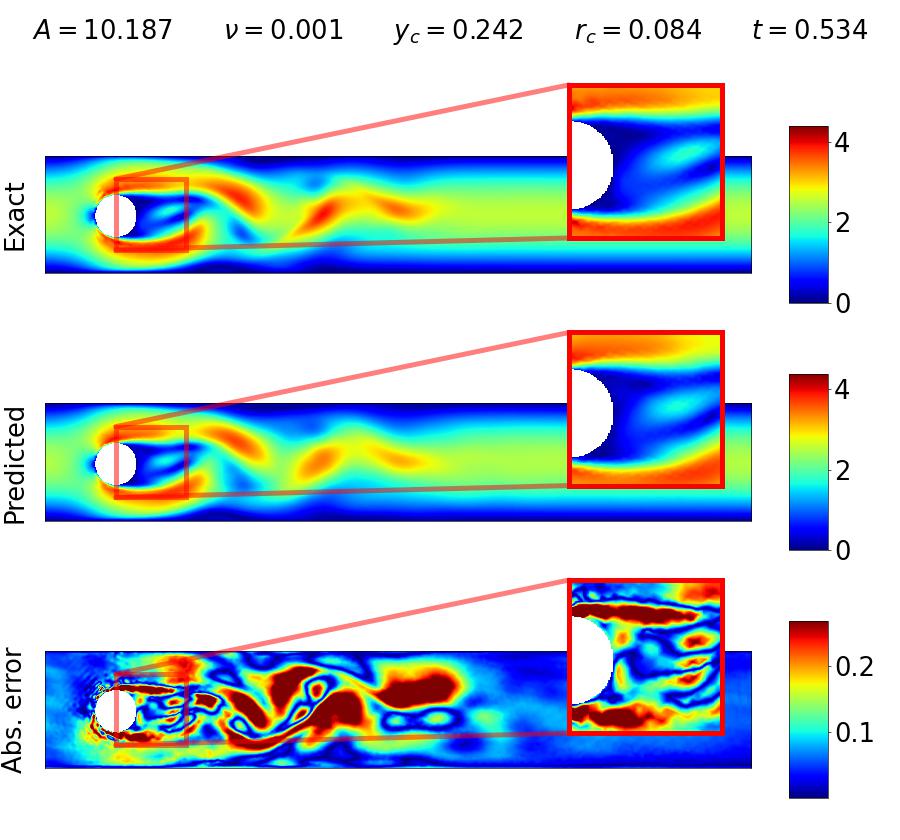}
    \includegraphics[width=0.49\textwidth]{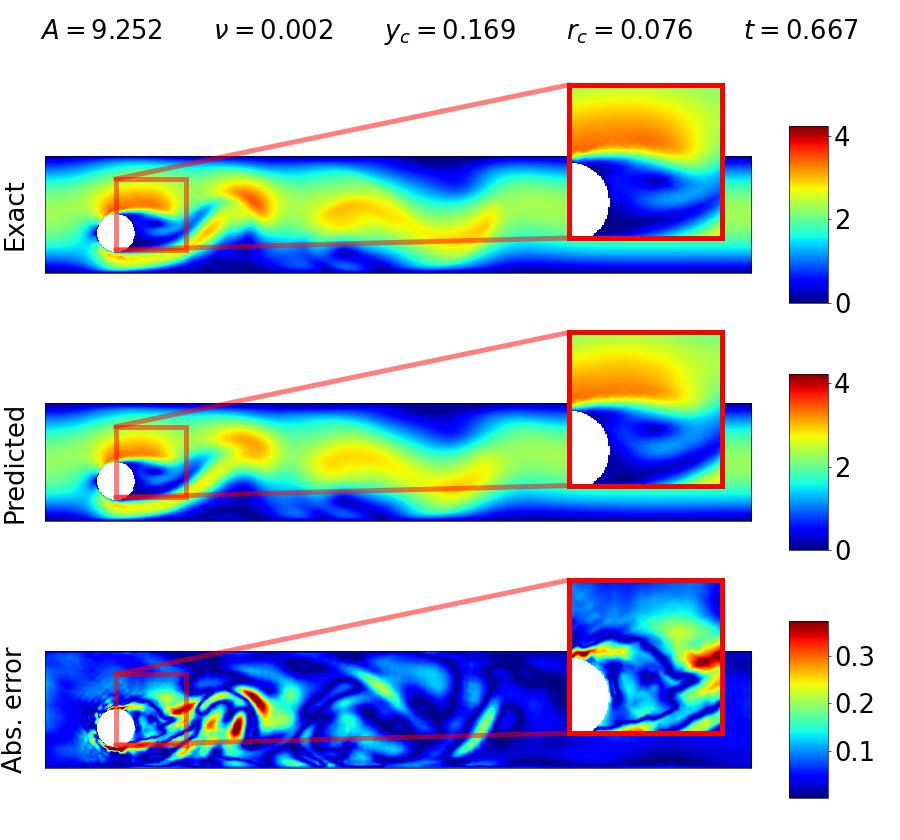}
    \caption{NS obstacle test case: accuracy of CGA-DL-ROM's prediction for two different test instances. By comparing the two instances, it is evident the large variability of the solution manifold entailed by the physical and geometrical parametrization.}
    \label{fig:ns_sample}
\end{figure}

\begin{figure}[htb!]
   \hspace{3.0cm}
   \includegraphics[width=0.55\textwidth]{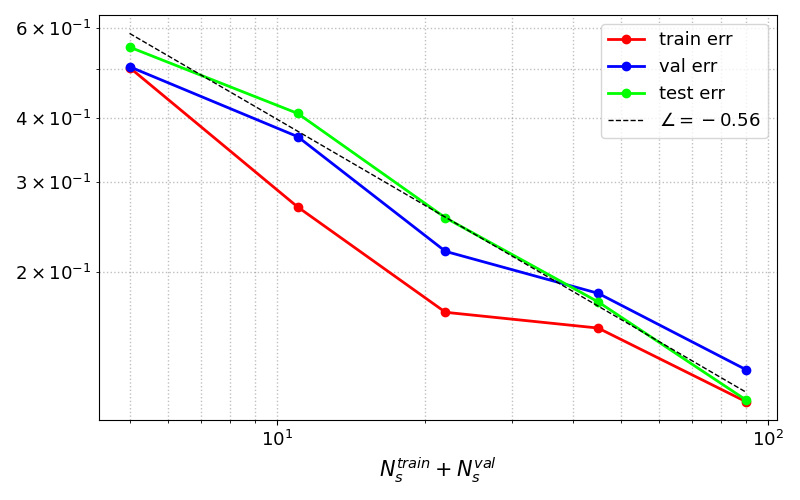}
  \caption{NS obstacle test case: ablation study with respect to the number of geometries available at training stage. The analysis is based upon the relative error metric $\mathcal{E}_R$.}
    \label{fig:ns-analysis-ablation-study}
\end{figure}

We train a CGA-DL-ROM for $65$ epochs, using a batch size of $15$ and a learning rate of $4 \times 10^{-4}$. We evaluate the accuracy of the trained network on a test set consisting of unseen geometries, obtaining $\mathcal{E}_R = 1.25 \times 10^{-1}$, which suggests a precise reconstruction of the parameter-to-solution map despite the fact that only $80$ different geometries are seen during the training phase. We also report in in Fig. \ref{fig:ns_sample} a visualization of the ROM prediction, highlighting that CGA-DL-ROM are able to suitably reconstruct even the low scale features occurring in the proximity of the obstacle's wake, which is a region of the domain that is heavily influenced by the geometrical parameters. We stress that Fig. \ref{fig:ns_sample} shows the prediction of two different test instances so that the reader can appreciate the accurate reconstruction of the solution manifold despite the extremely high variability entailed by the physical and geometrical parametrization.

Furthermore, we perform an ablation study with respect to the number of geometries that are available at the training stage. Specifically, we repeat the experiment hereinabove for $5$ times, halving the number of available training geometries at each step. For the sake of fairness, we use the same experiment configuration for all the tests, except for the number of training epochs, which is augmented to $100$ when $N_{s}^{train} + N_s^{val} \le 22$ to foster the convergence of the training algorithm. It is possible to observe from Fig. \ref{fig:ns-analysis-ablation-study} that the scaling law is nearly linear in the log-log space, suggesting that CGA-DL-ROMs efficiently take advantage of an increasing number of geometries to reach a better accuracy. Also, an insufficient number of geometries may lead to overfitting and often results in a non-negligible generalization gap (in this case, for $N_{s}^{train} + N_s^{val} \le 22$).

\subsection{Fluid temperature in a heat exchanger}
\label{subsec:heat-exchanger}
The purpose of this numerical experiment is to test the proposed framework with a relevant industrial benchmark involving a large-scale 3D setting. More precisely, we are interested in describing the temperature field of a fluid flowing through the heat exchanger. We remark that the present test case extends the numerical experiment proposed in \S 5 of \cite{brivio2024error} by introducing a geometrical parameterization. 

The heat exchanger shape $\Omega = \Omega(\xiv)$ is parametrized by a set of geometrical parameters $\xiv = \{\xi_k\}_{k=1}^{14}$ (for details, we refer the reader to Table \ref{table:heat-exchanger-geom-params}). The boundary of the heat exchanger is characterized by its three baffles $\Gamma_i(\xiv)$ (for $i=1,2,3$), the inlet $\Gamma_{IN}(\xiv)$, the outlet $\Gamma_{OUT}(\xiv)$ and the wall  $\Gamma_W(\xiv) = \partial\Omega(\xiv) \setminus (\cup_{j=1}^3 \Gamma_j(\xiv) \cup \Gamma_{IN}(\xiv) \cup \Gamma_{OUT}(\xiv))$.

 \begin{figure}[b!]
     \centering
     \includegraphics[width=0.65\linewidth]{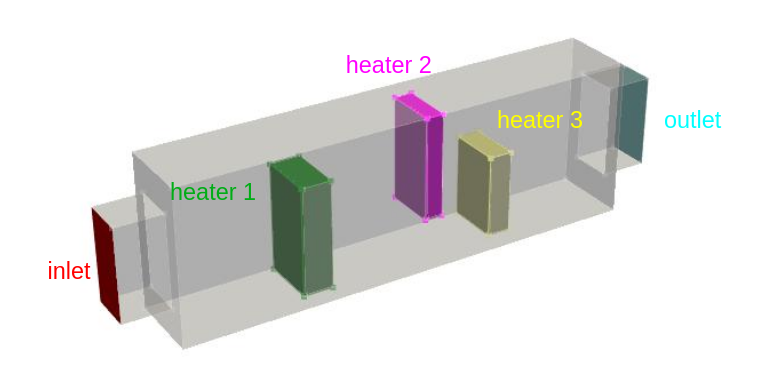}
     \caption{Heat exchanger test case: domain shape corresponding to an instance of the geometrical parameters.}
     \label{fig:enter-label}
 \end{figure}

\begin{table}[ht!]
\centering
\begin{tabular}{ |c|c|c| } 
\hline
Parameter & Description & Values \\ 
\hline
$\xi_1$ & $x$-position of the first heater & [0.5,0.9] \\ 
$\xi_2$ & length of the first heater & [0.1, 0.2] \\ 
$\xi_3$ & width of the first heater & [0.6, 0.7] \\ 
$\xi_4$ & height of the first heater & [0.3, 0.4] \\ 
$\xi_5$ & $x$-position of the second heater & [1.3, 1.7]  \\ 
$\xi_6$ & length of the second heater & [0.1, 0.2] \\ 
$\xi_7$ & width of the second heater & [0.6, 0.7] \\ 
$\xi_8$ & height of the first heater & [0.3, 0.4] \\ 
$\xi_9$ & $x$-position of the third heater & [2.1, 2.5] \\ 
$\xi_{10}$ & length of the third heater & [0.1, 0.2] \\ 
$\xi_{11}$ & width of the third heater & [0.45, 0.55] \\ 
$\xi_{12}$ & height of the third heater & [0.25, 0.35] \\ 
$\xi_{13}$ & width of the inlet and the outlet & [0.4, 0.6] \\ 
$\xi_{14}$ & height of the inlet and the outlet & [0.25, 0.35] \\ 
\hline
\end{tabular}
\caption{Heat exchanger test case: details on the geometrical parametrization}
\label{table:heat-exchanger-geom-params}
\end{table}

To generate the high-fidelity dataset, we model the temperature $u$ of the fluid with an advection-diffusion equation, namely,
\begin{equation}
    \left\{
    \begin{aligned}
        &\frac{\partial{u}}{\partial{t}} -  D \Delta u  + \bm{v} \cdot \nabla u = 0,& \mbox{in} \ \Omega(\xiv) \times (0,T] \\
        &u = \sum_{j=1}^3 g_j1_{\Gamma_j(\xiv)}, & \mbox{on} \  \cup_{j=1}^3\Gamma_j(\xiv) \times (0,T]\\
        &u = 0, & \mbox{on} \  (\Gamma_{IN}(\xiv) \cup \Gamma_{W}(\xiv)) \times (0,T]\\
        &\nabla u \cdot n = 0, & \mbox{on} \ \Gamma_{OUT}(\xiv) \times (0,T] \\
        &u_0 = 0, & \mbox{in} \  \Omega(\xiv),\\
    \end{aligned}
    \right.
\end{equation}
where $T=2$ is the time horizon, $D \in [0.01,0.1]$ is the thermal diffusivity of the fluid, while $g_1,g_2,g_3 \in [1,11]$ are the imposed temperatures on the corresponding baffle. We remark that $\bm{v}$ is the advection field due to the fluid flow, which is described by a stationary Navier-Stokes model, 
\begin{equation}
    \left\{
    \begin{aligned}
        & - \nu \Delta \bm{v} + (\bm{v} \cdot \nabla) \bm{v} + \nabla p = \mathbf{0} ,& \mbox{in} \ \Omega(\xiv)\\
         & \nabla \cdot \bm{v} = 0,& \mbox{in} \ \Omega\\
        &\bm{v} = [h(y,z;A),0,0], & \mbox{on} \  \Gamma_{IN}(\xiv)\\
        &\bm{v} = \mathbf{0}, & \mbox{on} \  \cup_{j=1}^3\Gamma_j \cup \Gamma_{W}(\xiv)\\
        &-p\bm{n} + \nu(\nabla\bm{v})\bm{n} = \mathbf{0}, & \mbox{on} \  \Gamma_{OUT}(\xiv),
    \end{aligned}
    \right.
\end{equation}
where $\nu$ is the viscosity of the fluid and
$$h(y,z;A) = 0.15^{-2} \cdot 16 A(0.75 - y)(y - 0.25)(0.4 - z)(z - 0.1).$$
represents the inlet profile regulated by the maximum amplitude $A \in [1,2]$. For the sake of readability, we collect the physical parameters in the vector $\muv = [A,\nu,g_1,g_2,g_3,D] \in \mathbb{R}^6$. We employ a three-step scheme for the computation of the high-fidelity solution, namely, for any parameter instance { \em{(i)} } we solve the Navier Stokes equations  with $\mathbb{P}^1\textnormal{b}-\mathbb{P}^1$ couple of Finite Element spaces, handling the nonlinearity with a Newton solver; we remark that we choose the Stokes solution on the same configuration as the initial guess of the nonlinear solver, { \em{(ii)} } we interpolate the velocity field onto a $\mathbb{P}^2$ space on the same mesh, and { \em{(iii)} } we solve the advection-diffusion problem with $\mathbb{P}^2$ elements. We emphasize that the entire synthetic data generation phase is implemented in \texttt{python}, using \texttt{FEniCS} \cite{langtangen2017solving} and \texttt{gmsh} \cite{gauzaine2017gmsh}, and takes $5h$ on a Intel Core i7 13th gen 32GB RAM laptop. 

\begin{figure}[h!]
    \centering
    \includegraphics[width=0.49\linewidth]{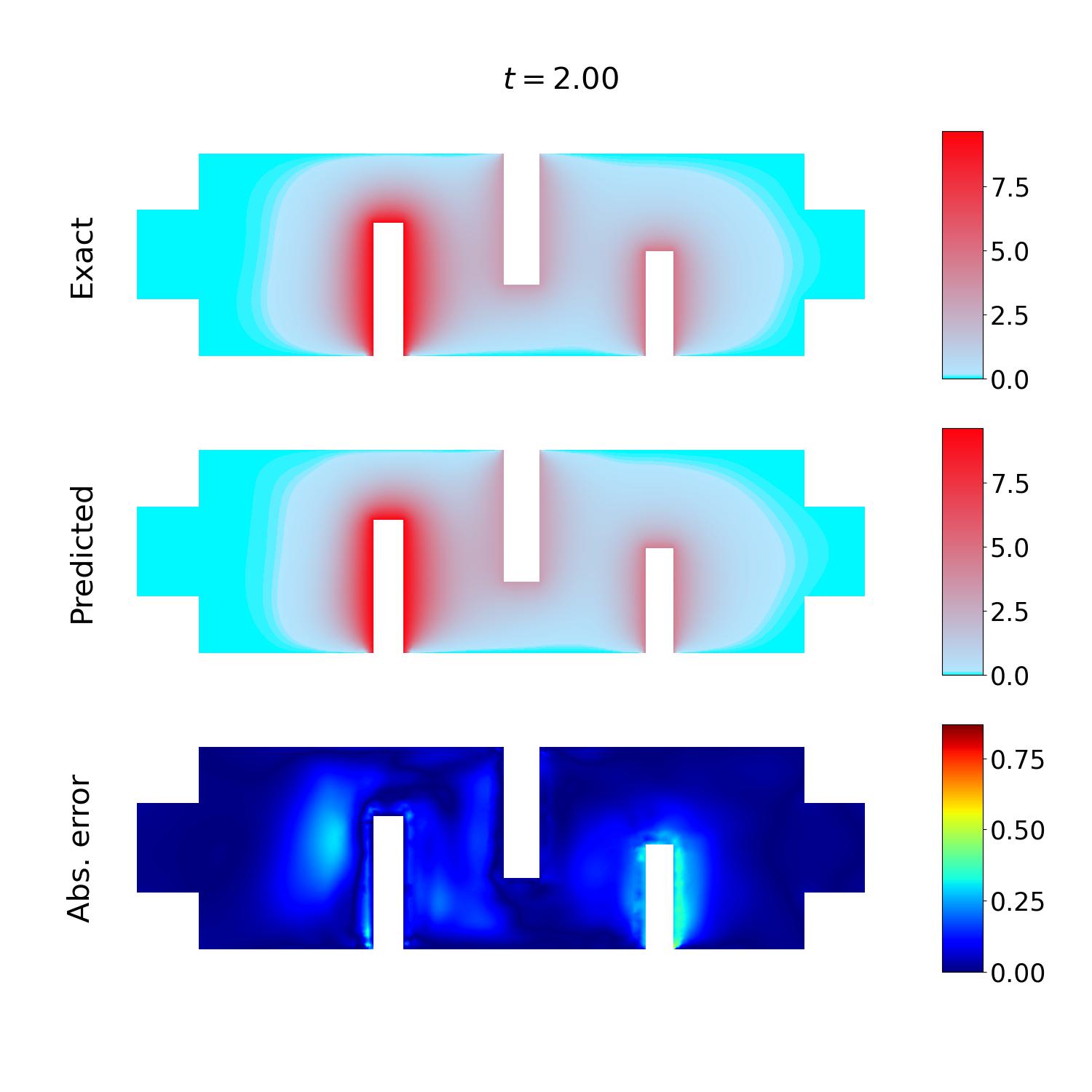}
    \includegraphics[width=0.49\linewidth]{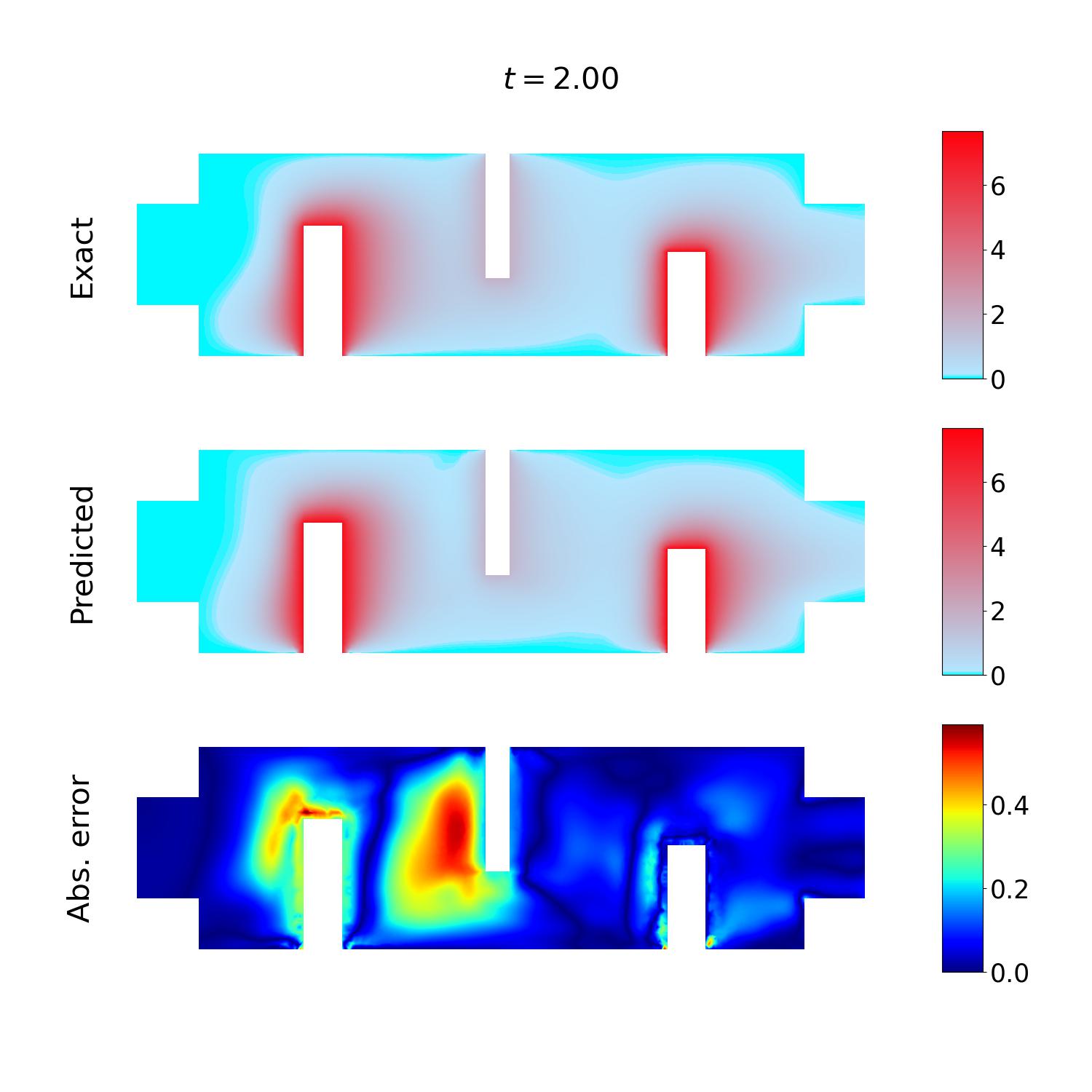}
    \caption{Heat exchanger test case: comparison between CGA-DL-ROM's prediction and ground truth on the plane $z=0.25$ at $t=2s$. We emphasize the notable effects of the geometrical parametrization on the domain shape.}
    \label{fig:mixer3d-visualization}
\end{figure}

\begin{figure}[h!]
    \centering
    \includegraphics[width=0.99\linewidth]{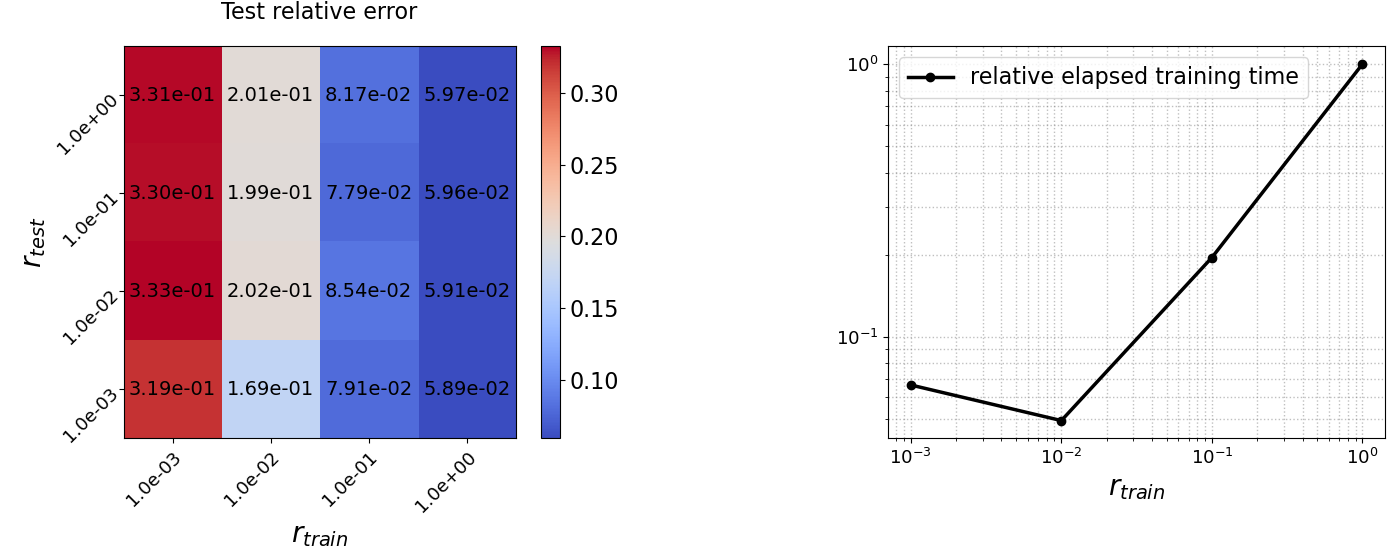}
    \caption{Heat exchanger test case: we show how training and test resolution influence the test accuracy {\em (left)}, and how the training resolution impacts on the computational training time {\em(right)}. }
    \label{fig:mixer3d-super-resolution-study}
\end{figure}

\begin{figure}[h!]
    \centering
    \subfloat[$r_{train} = 10^{-3}$]{
    \centering
    \includegraphics[width=0.24\linewidth]{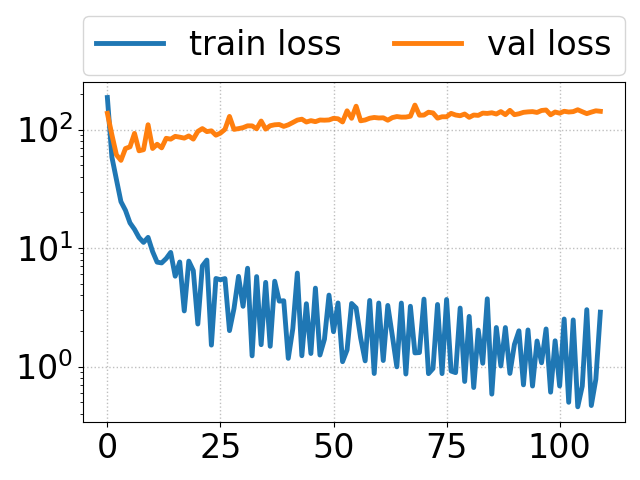}
    }
    \subfloat[$r_{train} = 10^{-2}$]{
    \centering
    \includegraphics[width=0.24\linewidth]{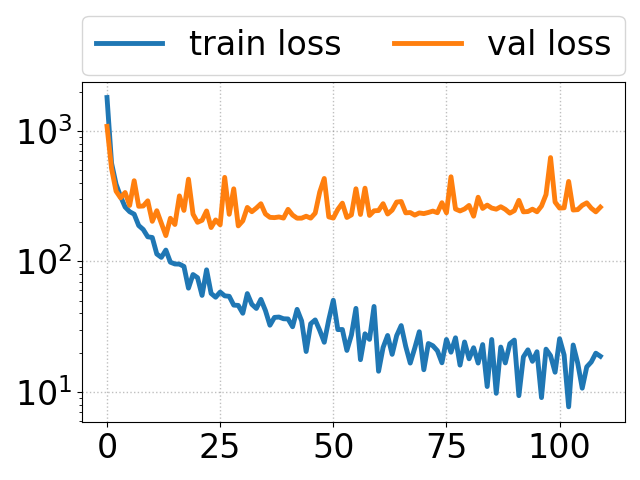}
    }
    \subfloat[$r_{train} = 10^{-1}$]{
    \centering
    \includegraphics[width=0.24\linewidth]{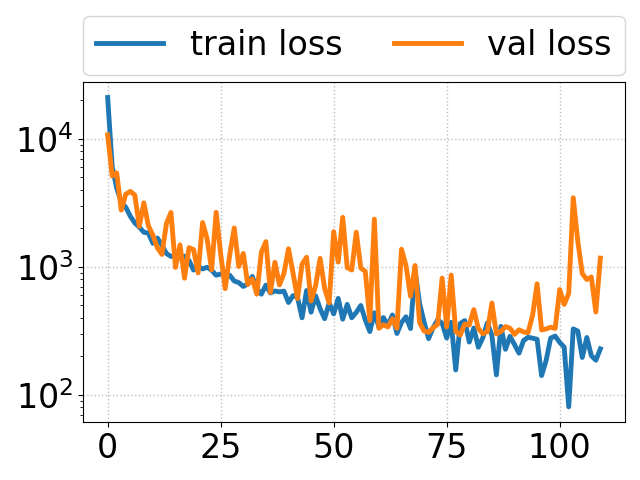}
    }
    \subfloat[$r_{train} = 10^{0}$]{
    \centering
    \includegraphics[width=0.24\linewidth]{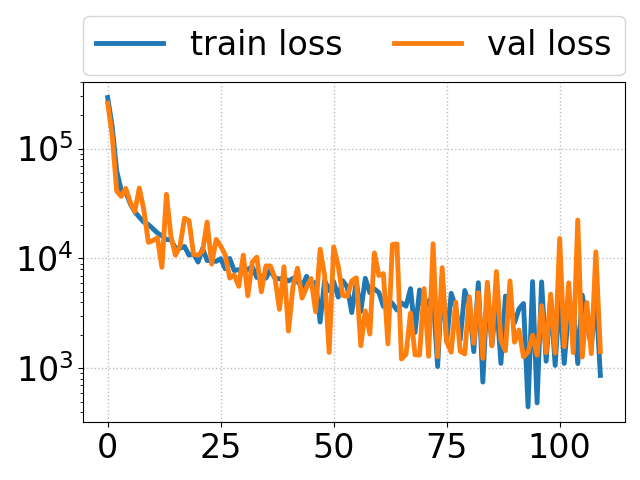}
    }
    \caption{Heat exchanger test case: training and validation loss decay for different values of the training resolution ratio $r_{train}$.}
    \label{fig:mixer-3d-loss-decay-various-resolutions}
\end{figure}

We train the architecture for a total of $110$ epochs with a learning rate of $10^{-3}$ and a batch size of $5$. Then, we test the trained CGA-DL-ROM on a test dataset consisting of unseen geometries, thus obtaining an accuracy of $\mathcal{E}_R = 5.97 \times 10^{-2}$. The qualitative results displayed in Fig. \ref{fig:mixer3d-visualization} further validate the approximation and generalization capabilities of CGA-DL-ROMs: indeed the proposed architecture is able to suitably capture the moving fronts of the temperature due to the imposed advection field. Nonetheless, we emphasize that the variability of the solution manifold is correctly reproduced as CGA-DL-ROM is capable of grasping the strong effect of the geometrical parametrization and the diffusive/advective behaviour.

Then, we emphasize that, thanks to its space-continuous architecture, CGA-DL-ROMs can be seamlessly trained and tested on different spatial resolutions. This property is convenient from a computational viewpoint as it allows us to train using data with a coarse resolution, thus alleviating the training computational burden (cf. Fig. \ref{fig:mixer3d-super-resolution-study}), and to evaluate it at a finer resolution. If the mismatch between the training and test resolutions does not result in a loss of accuracy, the method is said to have zero-shot super-resolution capabilities \cite{li2020fourier,kovachki2023neural}. Aiming at investigating such property, in our experimental setup, we vary the training and testing resolution by randomly extracting a fraction ($r_{train}, r_{test} \in (0,1]$) of the available dofs. Our experimental results (see Fig. \ref{fig:mixer3d-super-resolution-study}) demonstrate that CGA-DL-ROMs are capable of super-resolution because, for a fixed training resolution, the test error is independent of the test resolution. On the other hand, as expected, a finer training resolution fosters the convergence to more adequate minima in the loss functional. Indeed, we experimentally assess that a small training resolution prevents overfitting phenomena (cf. Fig. \ref{fig:mixer-3d-loss-decay-various-resolutions}), ultimately reducing the so-called generalization gap and resulting in an enhanced test accuracy. In conclusion, we advise domain practitioners to suitably calibrate the trade-off between the computational burden of the training phase and the desired test accuracy depending on the problem at hand.

\section*{Conclusion}
Within this work, we design and analyze the CGA-DL-ROM architecture to tackle problems featuring geometrical variability and parametrized domain shapes. Aside from the already investigated DL-ROM core, consisting of a nonlinear autoencoder and a reduced network, the main novelty of CGA-DL-ROMs consists of their characteristic CGA basis functions that make the proposed method:
\begin{itemize}
    \item \textit{space-continuous}: we cast CGA-DL-ROM in an infinite-dimensional framework, allowing for \textit{multi-resolution} datasets, which are quite common for problems featuring geometrical variability. Moreover, we complement the proposed framework with an abstract analysis of geometrically parametrized problems, and a characterization of the main properties, including the well-posedness of the dimensionality reduction problem;
    \item \textit{geometry-aware}: we inform such basis functions with the underlying geometrical parameters, thus creating a suitable inductive bias that strengthens the compressive capabilities, thus allowing domain practitioners to design light architectures with competitive prediction accuracy.
\end{itemize}

We showcase the versatility of the proposed framework through a series of numerical tests encompassing a diverse range of geometrically parametrized problems featuring generic \textit{multi-resolution} datasets, thus validating  CGA-DL-ROMs' remarkable approximation and generalization capabilities. Then, encouraged by the promising results of this novel framework on industrial benchmarks (cf. NS obstacle and Heat Exchanger test cases), in the future we aim at testing our proposed framework with real-world applications of even larger scale. Nonetheless, other lines of research may entail the analysis and the design of an extension of CGA-DL-ROM to deal with {\em (i)} either varying geometries on time, or {\em (ii)} non-parametric geometrical variability, where the domain shapes are parametrized by a set of unknown (possibly infinite-dimensional) geometrical parameters. Finally, we remark that the characterization of CGA basis functions is not limited to their use in the contest of DL-ROMs, and can be further extended to different architectures: this topic may constitute an additional, promising direction for future research.

\medskip

\section*{Acknowledgments}
The authors are members of the Gruppo Nazionale Calcolo Scientifico-Istituto Nazionale di Alta Matematica (GNCS-INdAM) and acknowledge the project “Dipartimento di Eccellenza” 2023-2027, funded by MUR. SF acknowledges the Istituto Nazionale di Alta Matematica (INdAM) for the financial support received through the “Concorso a n. 45 mensilità di borse di studio per l’estero per l’a.a. 2023-2024”. SB acknowledges the support of European Union - NextGenerationEU within the Italian PNRR program (M4C2, Investment 3.3) for the PhD Scholarship "Physics-informed data augmentation for machine learning applications". AM acknowledges the PRIN 2022 Project “Numerical approximation of uncertainty quantification problems for PDEs by multi-fidelity methods (UQ-FLY)” (No. 202222PACR), funded by the European Union - NextGenerationEU, and the Project “Reduced Order Modeling and Deep Learning for the real- time approximation of PDEs (DREAM)” (Starting Grant No. FIS00003154), funded by the Italian Science Fund (FIS) - Ministero dell'Università e della Ricerca. AM and SF acknowledge the project FAIR (Future Artificial Intelligence Research), funded by the NextGenerationEU program within the PNRR-PE-AI scheme (M4C2, Investment 1.3, Line on Artificial Intelligence). The authors thank Nicola R. Franco for the insightful discussions about parametric low-rank approximation.

\medskip
\printbibliography 

\newpage
\begin{appendices}

\section{Well-posedness of geometrically parametrized problems}
\label{sec:appendix-well-posedness}
We stress that the proof of well-posedness of problem \eqref{eq:general_formulation_reference} is case-dependent and goes beyond the purpose of the present work: we refer the reader to \cite{quarteroni2016redb} for more details on the subject.
For the sake of simplicity, here we assume that there exists a unique classical solution such that
\begin{equation*}
    \sup_{\tilde{\bx} \in \Omegaref} |\tilde{u}(\tilde{\bx}; t_1, \muv_1, \xiv_1) - \tilde{u}(\tilde{\bx}; t_2, \muv_2, \xiv_2)| \le \tilde{L} \|(t_1, \muv_1, \xiv_1) -(t_2, \muv_2, \xiv_2)\|  ,
\end{equation*}
for $\tilde{L} > 0$ and for any $(t_1, \muv_1, \xiv_1), (t_2, \muv_2, \xiv_2) \in \mathcal{P} \times \mathcal{T} \times \mathcal{G}$. Then, we employ a push-forward argument to recover the well-posedness of the original formulation entailed by problem \eqref{eq:general_formulation}. In particular, 
\begin{itemize}
    \item $\bx \mapsto u(\bx; t, \muv, \xiv)$ is continuous. Letting $L_Z = \max_{\xiv \in \mathcal{G}} Lip(Z(\cdot; \xiv))$ and choosing an arbitrary $\bx_1 \in \Omega(\xiv)$, it is straightforward to prove that 
    \begin{equation*}
        \forall \varepsilon > 0 \quad \exists \delta > 0 \quad : \quad \|\bx_1 - \bx_2\| < \delta/L_Z \quad \mbox{ such that } \quad |u(\bx_1; t, \muv, \xiv) - u(\bx_2; t, \muv, \xiv)| < \varepsilon
    \end{equation*}
    independently of $t, \muv, \xiv$. Indeed, since $\bx_i = Z^{-1}(\tilde{\bx}_i; \xiv)$ for $i=1,2$, we have
    \begin{equation*}
    \begin{aligned}
        |u(\bx_1; t, \muv, \xiv) - u(\bx_2; t, \muv, \xiv)| &= |u(Z^{-1}(\tilde{\bx}_1; \xiv); t, \muv, \xiv) - u(Z^{-1}(\tilde{\bx}_2; \xiv); t, \muv, \xiv)| \\
        & =|\tilde{u}(\tilde{\bx}_1; t, \muv, \xiv) - \tilde{u}(\tilde{\bx}_2; t, \muv, \xiv)|.
    \end{aligned}
    \end{equation*}
    Since $\|\tilde{\bx}_1 - \tilde{\bx}_2\| = \|Z(\bx_1; \xiv) - Z(\bx_2; \xiv)\| \le L_Z \|\bx_1 - \bx_2\| < \delta$, we conclude by continuity of $\bx \mapsto \tilde{u}(\bx; t, \muv, \xiv)$. 

    \item $(t,\muv,\xiv) \mapsto u(t,\muv,\xiv)$ is Lipschitz. Indeed, given an arbitrary choice of $(t_1, \muv_1, \xiv_1), (t_2, \muv_2, \xiv_2) \in \mathcal{P} \times \mathcal{T} \times \mathcal{G}$ and identifying  $\bx_1 = Z^{-1}(\tilde{\bx}; \xiv_1)$, $\bx_2 = Z^{-1}(\tilde{\bx}; \xiv_2)$ for any $\tilde{\bx} \in \Omegaref$, we have
    \begin{equation*}
    \begin{aligned}
        \sup_{\substack{\bx_1 \in \Omega(\xiv_1) \\ \bx_2 \in \Omega(\xiv_2)}} |u(\bx_1; t_1, \muv_1, \xiv_1) - u(\bx_2; t_2, \muv_2, \xiv_2)|  &= \sup_{\tilde{\bx} \in \Omegaref} |u(Z^{-1}(\tilde{\bx}; \xiv_1) ; t_1, \muv_1, \xiv_1) - u(Z^{-1}(\tilde{\bx}; \xiv_2); t_2, \muv_2, \xiv_2)|  \\
        &= \sup_{\tilde{\bx} \in \Omegaref} | \tilde{u}(\tilde{\bx} ; t_1, \muv_1, \xiv_1) - \tilde{u}(\tilde{\bx} ; t_2, \muv_2, \xiv_2)|  \\
        &\le \tilde{L} \|(t_1, \muv_1, \xiv_1) -(t_2, \muv_2, \xiv_2)\|,
    \end{aligned}
    \end{equation*}
    which ensures continuous dependence from data.

\end{itemize}

\section{Analysis of the morphing operators}
\label{sec:appendix-morphing}
\begin{lemma}
It is possible to prove that
\begin{equation*}
\begin{array}{rlccccl}
    &\mathcal{Z}_{\xiv} &: &L^2(\Omegaref) \hspace{0.3cm} &\rightarrow &L^2_\zeta(\Omega(\xiv)) &\mbox{ as } \tilde{f} \mapsto \tilde{f} \circ Z(\xiv) \\
    &\mathcal{Z}_{\xiv}^{-1} &: &L^2_\zeta(\Omega(\xiv)) &\rightarrow &L^2(\Omegaref) &\mbox{ as } f \mapsto f \circ Z^{-1}(\xiv),
\end{array}
\end{equation*} 
are linear bounded functionals.
\end{lemma}
\begin{proof}
Let $\xiv \in \mathcal{G}$. Linearity of both $\mathcal{Z}_{\xiv}$ and $\mathcal{Z}_{\xiv}^{-1}$ is trivial to prove thanks to their definition. Then, owing to the change of variable formula \eqref{eq:change-of-variable}, we have that
\begin{itemize}
    \item $\mathcal{Z}_{\xiv}$ is bounded. Indeed,
    \begin{equation*}
        \|\mathcal{Z}_{\xiv}\|_{\star} = \sup_{\substack{\tilde{f} \in L^2(\Omegaref): \\\|\tilde{f}\|_{L^2(\Omegaref)} \not= 0}} \frac{\|\mathcal{Z}_{\xiv}(\tilde{f})\|_{L^2(\Omega(\xiv))}}{\|\tilde{f}\|_{L^2(\Omegaref)}} = \sup_{\substack{f \in L^2(\Omega(\xiv)): \\\|f\|_{L^2(\Omega(\xiv))} \not= 0}} \frac{\|\tilde{f}\|_{L^2(\Omegaref)}}{\|\tilde{f}\|_{L^2(\Omegaref)}} \le 1.
    \end{equation*}
    \item $\mathcal{Z}_{\xiv}^{-1}$ is bounded. Indeed,
    \begin{equation*}
        \|\mathcal{Z}_{\xiv}^{-1}\|_{\star} = \sup_{\substack{f \in L_{\zeta}^2(\Omega(\xiv)): \\\|f\|_{L_{\zeta}^2(\Omega(\xiv))} \not= 0}} \frac{\|\mathcal{Z}_{\xiv}^{-1}(f)\|_{L^2(\Omegaref)}}{\|f\|_{L_{\zeta}^2(\Omega(\xiv))}} = \sup_{\substack{f \in L_{\zeta}^2(\Omega(\xiv)): \\\|f\|_{L_{\zeta}^2(\Omega(\xiv))} \not= 0}} \frac{\|f\|_{L_{\zeta}^2(\Omega(\xiv))}}{\|f\|_{L_{\zeta}^2(\Omega(\xiv))}} \le 1.
    \end{equation*}
\end{itemize}
\end{proof}

\section{Supplementary proofs}
\label{sec:appendix-proofs}

\subsection{Proof of Lemma \ref{lemma:CGA-functional}}
\begin{proof}
    The proof follows similar arguments as the proof of \textit{Lemma 2.2} and \textit{Theorem 4.2} of \cite{franco2024measurabilitycontinuityparametriclowrank} and is organized into three parts:
    \begin{itemize}
        \item[(i)] $B$ is weakly compact since $L^2(\tilde{\Omega})$ is reflexive. Then, it is possible to endow $B$ with a metric $d_B$ that makes it compatible with the weak topology (see \cite{brezis2010functional}, Proposition (3.29)).
        \item[(ii)] Moreover, we show that $\mathcal{J}_{CGA}:\mathcal{G} \times [B]^N \rightarrow \mathbb{R}$ is weakly lower semi-continuous with respect to the product topology on $\mathcal{G} \times [B]^N$.
        \item[(iii)] Since a weakly lower semi-continuous functional attains its infimum on a weakly compact metric space \cite{sahney1982best}, we conclude.
    \end{itemize}
    We are left to prove that $\mathcal{J}_{CGA}:\mathcal{G} \times [B]^N \rightarrow \mathbb{R}$ is lower semi-continuous. To this end, consider  the sequences $\xiv^{(k)} \rightarrow \xiv$ and $w_n^{(k)} \rightharpoonup w_n$, for $n=1,\ldots,N$. Now, for the sake of readability we omit the dependencies of $u$ on $\muv,t$. Thanks to $\xiv \mapsto u(\xiv)$ being Lipschitz-continuous and $\mathcal{Z}_{\xiv}^{-1}$ being bounded for any $\xiv \in \mathcal{G}$, then $\xiv \mapsto \tilde{u}(\xiv) = \mathcal{Z}_{\xiv}^{-1}(u(\xiv))$ is continuous. Thus, we obtain $\tilde{u}(\xiv^{(k)}) \rightarrow \tilde{u}(\xiv)$ and $\tilde{u}(\xiv^{(k)}) \rightharpoonup \tilde{u}(\xiv)$. Then, it is straightforward to see that for any $n=1,\ldots,N$,
    \begin{itemize}
        \item $(\tilde{u}(\xiv^{(k)}),w_n^{(k)})_{L^2(\Omegaref)} \rightarrow (\tilde{u}(\xiv),w_n)_{L^2(\Omegaref)}$, since
        \begin{equation*}
        \begin{aligned}
            &(\tilde{u}(\xiv^{(k)}),w_n^{(k)})_{L^2(\Omegaref)} -(\tilde{u}(\xiv),w_n)_{L^2(\Omegaref)} =\\
            & = (\tilde{u}(\xiv^{(k)}),w_n^{(k)})_{L^2(\Omegaref)} - (\tilde{u}(\xiv),w_n^{(k)})_{L^2(\Omegaref)} + (\tilde{u}(\xiv),w_n^{(k)}) - (\tilde{u}(\xiv),w_n)_{L^2(\Omegaref)} \\
            &\le \|w_n^{(k)}\|_{L^2(\Omegaref)} \|\tilde{u}(\xiv^{(k)}) - \tilde{u}(\xiv)\|_{L^2(\Omegaref)} + (\tilde{u}(\xiv), w_n^{(k)} - w_n)_{L^2(\Omegaref)}\rightarrow 0
        \end{aligned}
        \end{equation*}
        thanks to $\{w_n^{(k)}\}_k$ being weak convergent (thus bounded).
        \item $(\tilde{u}(\xiv^{(k)}),w_n^{(k)})_{L^2(\Omegaref)}w_n^{(k)} \rightharpoonup (\tilde{u}(\xiv),w_n)_{L^2(\Omegaref)}w_n$, since for any $z \in L^2(\Omegaref)$
        \begin{equation*}
        \begin{aligned}
            &((\tilde{u}(\xiv^{(k)}),w_n^{(k)})_{L^2(\Omegaref)}w_n^{(k)},z)_{L^2(\Omegaref)} =\\
            &=  (\tilde{u}(\xiv^{(k)}),w_n^{(k)})_{L^2(\Omegaref)}(w_n^{(k)},z)_{L^2(\Omegaref)}  \rightarrow (\tilde{u}(\xiv),w_n)_{L^2(\Omegaref)} (w_n^,z)_{L^2(\Omegaref)} = ((\tilde{u}(\xiv),w_n)_{L^2(\Omegaref)}w_n,z)_{L^2(\Omegaref)}
        \end{aligned}
        \end{equation*}
    \end{itemize}
    Thus, we have that 
    \begin{equation*}
        \sum_{n=1}^{N} (\tilde{u}(\xiv^{(k)}),w_n^{(k)})_{L^2(\Omegaref)}w_n^{(k)} \rightharpoonup \sum_{n=1}^N (\tilde{u}(\xiv),w_n)_{L^2(\Omegaref)}w_n
    \end{equation*}
    and by lower semi-continuity of $\|\cdot\|_{L^2(\Omegaref)}$ with respect to the weak topology it is possible to prove that
    \begin{equation*}
        \liminf_{k\rightarrow \infty} \| \tilde{u}(\xiv^{(k)}) - \sum_{n=1}^{N} (\tilde{u}(\xiv^{(k)}),w_n^{(k)})_{L^2(\Omegaref)}w_n^{(k)} \|_{L^2(\Omegaref)} \ge \| \tilde{u}(\xiv) - \sum_{n=1}^{N} (\tilde{u}(\xiv),w_n)_{L^2(\Omegaref)}w_n \|_{L^2(\Omegaref)},
    \end{equation*}
    which concludes the proof.
\end{proof}

\subsection{Proof of Proposition \ref{prop:nn-appx}}
\begin{proof}
Fix $N < \infty$. Define $\Omega_{\mathcal{G}} := \overline{\cup_{\xiv \in \mathcal{G}} \Omega(\xiv)}$ and
\begin{equation*}
    \varepsilon^* := \min\bigg\{2N\|\mathcal{Z}_{\xiv}^{-1}(u)\|_{L^2(\mathcal{T} \times \mathcal{P} \times \mathcal{G}; L^2(\Omegaref))},\frac{1}{(2 \mathsf{BAE}_{CGA}(N)^{1/2})}\bigg\}.
\end{equation*} 
Then, for any $\xiv \in \mathcal{G}$, we define the zero-extension operator $\epsilon_{\xiv}$ as
\begin{equation*}
    \epsilon_{\xiv} (\mathcal{Z}_{\xiv}(v_n^{CGA}(x,\xiv))):=
    \left\{
    \begin{array}{ll}
         \mathcal{Z}_{\xiv}(v_n^{CGA}(\bx,\xiv)), & \bx \in \Omega(\xiv) \\
         0, &\bx \in \Omega_{\mathcal{G}} \setminus \Omega(\xiv),
    \end{array}
    \right.
\end{equation*}
for any $n=1,\ldots,N$. We immediately notice that $\epsilon_{\xiv} (\mathcal{Z}_{\xiv}(v_n^{CGA})) \in W^{1,2}(\Omega_{\mathcal{G}}) \times \mathcal{G}$. Indeed, setting $m = \min_{\xiv \in \mathcal{G}} \min_{\bx \in \Omega(\xiv)} \zeta(\bx;\xiv) > 0$, we have
\begin{equation*}
\begin{aligned}
      m \|\epsilon_{\xiv}(\mathcal{Z}_{\xiv}(w))\|_{L^2(\Omega_{\mathcal{G}} \times \mathcal{G})}^2 &=
      m \int_{\Omega_{\mathcal{G}} \times \mathcal{G}} \epsilon_{\xiv}(\mathcal{Z}_{\xiv}(w(\bx;\xiv)))^2 d\bx d\xiv \\
      &= m \int_{\mathcal{G}} \int_{\Omega(\xiv)} \epsilon_{\xiv}(\mathcal{Z}_{\xiv}(w(\bx;\xiv)))^2 d\bx d\xiv \\
      &\le \int_{\mathcal{G}} \int_{\Omega(\xiv)} \zeta(\bx;\xiv) \mathcal{Z}_{\xiv}(w(\bx;\xiv))^2 d\bx d\xiv \\
      &= \int_{\mathcal{G}} \int_{\Omegaref}  w(\tilde{\bx};\xiv)^2 d\tilde{\bx} d\xiv = \|w\|_{L^2(\Omegaref \times \mathcal{G})}^2 < \infty,
\end{aligned}
\end{equation*}
where $w$ stands for either $v_n^{CGA}$ or its weak derivative.
Then, thanks to G\"urhing's Theorem \cite{guhring2021appx}, there exists a set of neural networks $(\bx,\xiv) \mapsto \{\hat{v}^{CGA}_n(\bx, \xiv)\}_{n=1}^N$ such that
\begin{equation*}
    \|\hat{v}^{CGA}_n - \epsilon_{\xiv} (\mathcal{Z}_{\xiv}(v_n^{CGA})\|_{L^2(\Omega_{\mathcal{G}} \times \mathcal{G})} < \frac{\varepsilon^2}{2MN\|\mathcal{Z}_{\xiv}^{-1}(u)\|_{L^2(\mathcal{T} \times \mathcal{P} \times \mathcal{G}; L^2(\Omegaref))}}.
\end{equation*}
where $M = \sup_{\xiv \in \mathcal{G}} \sup_{\bx \in \Omega(\xiv)} \zeta(\bx;\xiv)$. We note that
\begin{equation*}
\begin{aligned}
     M\|\hat{v}^{CGA}_n - \epsilon_{\xiv} (\mathcal{Z}_{\xiv}(v_n^{CGA}))\|_{L^2(\Omega_{\mathcal{G}} \times \mathcal{G})} &= M\|\hat{v}^{CGA}_n - \epsilon_{\xiv} (\mathcal{Z}_{\xiv}(v_n^{CGA})\|_{L^2(\mathcal{G}; L^2(\Omega_{\mathcal{G}})) }\\
     &\ge  \|\hat{v}^{CGA}_n - \epsilon_{\xiv} (\mathcal{Z}_{\xiv}(v_n^{CGA})\|_{L^2(\mathcal{G}; L_{\zeta}^2(\Omega(\xiv)))} \\
     & =  \|\hat{v}^{CGA}_n - \mathcal{Z}_{\xiv}(v_n^{CGA})\|_{L^2(\mathcal{G}; L_{\zeta}^2(\Omega(\xiv)))} \\
     & =  \|\mathcal{Z}_{\xiv}^{-1}(\hat{v}^{CGA}_n) - v_n^{CGA}\|_{L^2(\mathcal{G}; L^2(\Omegaref))} \\
     & = \|\mathcal{Z}_{\xiv}^{-1}(\hat{v}^{CGA}_n) - v_n^{CGA}\|_{L^2(\Omegaref \times \mathcal{G})},
\end{aligned}
\end{equation*}
so that
\begin{equation*}
    \|\mathcal{Z}_{\xiv}^{-1}(\hat{v}^{CGA}_n) - v_n^{CGA}\|_{L^2(\Omegaref \times \mathcal{G})} < \frac{\varepsilon^2}{2N\|\mathcal{Z}_{\xiv}^{-1}(u)\|_{L^2(\mathcal{T} \times \mathcal{P} \times \mathcal{G}; L^2(\Omegaref))}}.
\end{equation*}
The whole architecture is obtained by stacking together the neural networks $\{\hat{v}^{CGA}_n\}_{n=1}^N$. Thus, by employing in sequence the lower triangular inequality, the triangular inequality, and the Cauchy-Schwarz inequality, we obtain
\begin{equation*}
\begin{aligned}
&\bigg[\int_{\mathcal{G}} \mathcal{J}_{CGA}(\xiv, \{\mathcal{Z}^{-1}_{\xiv}(\hat{v}^{CGA}_n(\xiv))\}_{n=1}^{N}) \bigg]^{1/2} - \bigg[\mathsf{BAE}_{CGA}(N)\bigg]^{1/2} = \\
&\overset{def}{=} \bigg\|\mathcal{Z}_{\xiv}^{-1}(u) - \sum_{n=1}^N(\mathcal{Z}_{\xiv}^{-1}(u),\mathcal{Z}_{\xiv}^{-1}(\hat{v}^{CGA}_n))_{L^2(\Omegaref)} \mathcal{Z}_{\xiv}^{-1}(\hat{v}^{CGA}_n)\bigg\|_{L^2(\mathcal{T} \times    \mathcal{P} \times \mathcal{G}; L^2(\Omegaref))} - \\
&\hspace{3.5cm} - \bigg\|\mathcal{Z}_{\xiv}^{-1}(u) - \sum_{n=1}^N(\mathcal{Z}_{\xiv}^{-1}(u),\mathcal{Z}_{\xiv}^{-1}(v^{CGA}_n))_{L^2(\Omegaref)} \mathcal{Z}_{\xiv}^{-1}(v^{CGA}_n)\bigg\|_{L^2(\mathcal{T} \times    \mathcal{P} \times \mathcal{G}; L^2(\Omegaref))}  \\
& \overset{LTI}{\le}  \bigg\|\sum_{n=1}^N \bigg[(\mathcal{Z}_{\xiv}^{-1}(u),\mathcal{Z}_{\xiv}^{-1}(\hat{v}^{CGA}_n))_{L^2(\Omegaref)} \mathcal{Z}_{\xiv}^{-1}(\hat{v}^{CGA}_n) \\
& \hspace{3.5cm} - (\mathcal{Z}_{\xiv}^{-1}(u),\mathcal{Z}_{\xiv}^{-1}(v^{CGA}_n))_{L^2(\Omegaref)} \mathcal{Z}_{\xiv}^{-1}(v^{CGA}_n) \bigg]\bigg\|_{L^2(\mathcal{T} \times    \mathcal{P} \times \mathcal{G}; L^2(\Omegaref))} \\
& \hspace{1.5mm} = \bigg\|\sum_{n=1}^N \bigg[(\mathcal{Z}_{\xiv}^{-1}(u),\mathcal{Z}_{\xiv}^{-1}(\hat{v}^{CGA}_n - v^{CGA}_n))_{L^2(\Omegaref)} \mathcal{Z}_{\xiv}^{-1}(\hat{v}^{CGA}_n) \\
& \hspace{3.5cm}- (\mathcal{Z}_{\xiv}^{-1}(u),\mathcal{Z}_{\xiv}^{-1}(v^{CGA}_n))_{L^2(\Omegaref)} \mathcal{Z}_{\xiv}^{-1}(v^{CGA}_n - \hat{v}^{CGA}_n) \bigg]\bigg\|_{L^2(\mathcal{T} \times    \mathcal{P} \times \mathcal{G}; L^2(\Omegaref))} \\
& \overset{TI}{\le} \sum_{n=1}^N \bigg\| (\mathcal{Z}_{\xiv}^{-1}(u),\mathcal{Z}_{\xiv}^{-1}(\hat{v}^{CGA}_n - v^{CGA}_n))_{L^2(\Omegaref)} \mathcal{Z}_{\xiv}^{-1}(\hat{v}^{CGA}_n) \bigg\|_{L^2(\mathcal{T} \times    \mathcal{P} \times \mathcal{G}; L^2(\Omegaref))} + \\\
& \hspace{3.5cm} \bigg\| (\mathcal{Z}_{\xiv}^{-1}(u),\mathcal{Z}_{\xiv}^{-1}(v^{CGA}_n))_{L^2(\Omegaref)} \mathcal{Z}_{\xiv}^{-1}(v^{CGA}_n - \hat{v}^{CGA}_n) \bigg\|_{L^2(\mathcal{T} \times    \mathcal{P} \times \mathcal{G}; L^2(\Omegaref))} \\
& \overset{CS}{\le} \sum_{n=1}^N \|\mathcal{Z}_{\xiv}^{-1}(u)\|_{L^2(\mathcal{T} \times \mathcal{P} \times \mathcal{G}; L^2(\Omegaref))} \|v^{CGA}_n -\hat{v}^{CGA}_n\|_{L^2(\Omegaref \times \mathcal{G})}\bigg[\|v^{CGA}_n\|_{L^2(\Omegaref \times \mathcal{G})} + \|\hat{v}^{CGA}_n\|_{L^2(\Omegaref \times \mathcal{G})}\bigg]
\end{aligned}
\end{equation*}
Now, we observe that
\begin{equation*}
\begin{aligned}
    \|\hat{v}^{CGA}_n\|_{L^2(\Omegaref \times \mathcal{G})} &\le \|v^{CGA}_n -\hat{v}^{CGA}_n\|_{L^2(\Omegaref \times \mathcal{G})} + \|v^{CGA}_n\|_{L^2(\Omegaref \times \mathcal{G})} \\
    &< \frac{\varepsilon^2}{2N\|\mathcal{Z}_{\xiv}^{-1}(u)\|_{L^2(\mathcal{T} \times \mathcal{P} \times \mathcal{G}; L^2(\Omegaref))}} + 1 \\
    &\le \frac{(\varepsilon^*)^2}{2N\|\mathcal{Z}_{\xiv}^{-1}(u)\|_{L^2(\mathcal{T} \times \mathcal{P} \times \mathcal{G}; L^2(\Omegaref))}} + 1 \le 2,
\end{aligned}
\end{equation*}
which allows us to obtain the bound
\begin{equation*}
    \bigg[\int_{\mathcal{G}} \mathcal{J}_{CGA}(\xiv, \{\mathcal{Z}^{-1}_{\xiv}(\hat{v}^{CGA}_n(\xiv))\}_{n=1}^{N}) \bigg]^{1/2} - \bigg[\mathsf{BAE}_{CGA}(N)\bigg]^{1/2} < \varepsilon^2.
\end{equation*}
By isolating the first addendum of the l.h.s. and by squaring both sides of the latter equation, we can conclude
\begin{equation*}
    \int_{\mathcal{G}} \mathcal{J}_{CGA}(\xiv, \{\mathcal{Z}^{-1}_{\xiv}(\hat{v}^{CGA}_n(\xiv))\}_{n=1}^{N})  < \varepsilon^4 + 2 \varepsilon^2  \bigg[\mathsf{BAE}_{CGA}(N)\bigg]^{1/2} + \mathsf{BAE}_{CGA}(N) < \varepsilon +\mathsf{BAE}_{CGA}(N),
\end{equation*}
since $\varepsilon < \varepsilon^* \le 1/(2 [\mathsf{BAE}_{CGA}(N)]^{1/2})$.
\end{proof}
\end{appendices}

\end{document}